\documentclass[a4paper, 10pt, draft]{amsart}
\usepackage{amsmath}
\usepackage{amssymb}
\usepackage{amsthm}
\usepackage{comment}
\usepackage{shuffle}
\usepackage{enumerate}
\newtheorem{theorem}{Theorem}
\newtheorem*{theorem*}{Theorem}
\newtheorem{lemma}{Lemma}
\newtheorem*{lemma*}{Lemma}

\newtheorem*{corollary*}{Corollary}
\newtheorem{proposition}{Proposition}
\newtheorem*{proposition*}{Proposition}

\theoremstyle{definition}
\newtheorem{definition}{Definition}
\newtheorem*{definition*}{Definition}
\newtheorem{remark}{Remark}
\newtheorem*{remark*}{Remark}

\newtheorem*{example*}{Example}
\newtheorem{conjecture}{Conjecture}
\newtheorem*{conjecture*}{Conjecture}
\allowdisplaybreaks[1]

\DeclareMathOperator{\Li}{Li}
\DeclareMathOperator{\Ls}{Ls}
\DeclareMathOperator{\SLs}{SLs}
\DeclareMathOperator{\Cl}{Cl}
\DeclareMathOperator{\Gl}{Gl}

\begin{document}
\title[Evaluation of iterated log-sine integrals]{Evaluation of iterated log-sine integrals in terms of multiple polylogarithms}
\author{RYOTA UMEZAWA}
\date{}
\maketitle

\begin{abstract}
It is known that multiple zeta values can be written in terms of certain iterated log-sine integrals. Conversely, we evaluate iterated log-sine integrals in terms of multiple polylogarithms and multiple zeta values in this paper. We also suggest some conjectures on multiple zeta values, multiple Clausen values, multiple Glaisher values and iterated log-sine integrals. 
\end{abstract}
\section{Introduction}
For an index $\mathbf{k} = (k_{1},\dots,k_{n})$, we define the weight of $\mathbf{k}$ by $|\mathbf{k}| = k_{1}+\dots+k_{n}$, and the depth of $\mathbf{k}$ by ${\rm dep}(\mathbf{k})=n$. We call an index $\mathbf{k}=(k_{1},\dots,k_{n})$ \textit{admissible index} if $\mathbf{k}\in\mathbb{N}^{n}$ and $k_{n}\ge2$. We regard $\emptyset$ as the admissible index of weight $0$ and depth $0$. For two indices $\mathbf{k} = (k_{1},\dots,k_{n})$ and $\mathbf{l} = (l_{1},\dots,l_{n})$, we define
\begin{align*}
\mathbf{k}\pm\mathbf{l} = (k_{1}\pm l_{1},\dots,k_{n}\pm l_{n}).
\end{align*}
We use the notation $\{k\}^{n}$ for $n$ repetitions of $k$. For example, $(1,1,2) = (\{1\}^{2},2)$.

Iterated log-sine integrals are the following integrals introduced by the author \cite{U}.
\begin{definition}[Iterated log-sine integrals]
For $\sigma \in \mathbb{R}$, $\mathbf{k} = (k_{1}, \dots, k_{n}) \in \mathbb{N}^{n}$ and $\mathbf{l} = (l_{1},\dots,l_{n}) \in (\mathbb{Z}_{\ge 0})^{n}$, we define
\[\Ls_{\mathbf{k}}^{\mathbf{l}}(\sigma)=(-1)^{n}\int_{0}^{\sigma}\int_{0}^{\theta_{n}}\dots\int_{0}^{\theta_{2}}
\prod_{u=1}^{n}\theta_u^{l_u}\left(A(\theta_{u})\right)^{k_u-1-l_u}\,d\theta_{1}\cdots d\theta_{n},\]
where $A(\theta) = \log\left|2\sin(\theta/2)\right|$. We understand $\Ls_{\emptyset}^{\emptyset}(\sigma)=1$.
\end{definition}
If $k_{u} -1-l_{u} \ge 0$ for all $u \in \{1,\dots,n\}$, then this integral converges absolutely for any $\sigma \in \mathbb{R}$. 
In the case $n=1$, $\Ls_{k}^{(l)}(\sigma)$ is called (generalized) log-sine integrals.
It was shown that iterated log-sine integrals are related to multiple zeta values and multiple polylogarithms by the author \cite{U}. In particular, all multiple zeta values can be written as a $\mathbb{Q}$-linear combination of products of $\pi^m\ (m \ge 0)$ and an iterated log-sine integral at $\pi/3$ satisfying $k_{u} -1 - l_{u} > 0$ for all $u \in \{1,\dots,n\}$. Here, multiple zeta values and multiple polylogarithms are defined as follows, respectively:
\begin{definition}[Multiple zeta values]
For an admissible index $\mathbf{k} = (k_{1},\dots,k_{n})$, we define
\[\zeta(\mathbf{k}) = \sum_{0 < m_1  < \dots < m_n} \frac{1}{m_{1}^{k_{1}} \cdots m_{n}^{k_{n}}}.\]
We understand $\zeta(\emptyset)=1$.
\end{definition}
\begin{definition}[Multiple polylogarithms]
For $\mathbf{k} = (k_{1},\dots,k_{n}) \in \mathbb{N}^{n}$, we define
\[\mathrm{Li}_{\mathbf{k}}(z) = \sum_{0<m_{1} < \dots < m_{n}} \frac{z^{m_{n}}}{m_{1}^{k_{1}} \cdots m_{n}^{k_{n}}}.\]
We understand $\Li_{\emptyset}(z)=1$.
\end{definition}
Multiple polylogarithms converge absolutely when $|z| < 1$, and can be continued holomorphically to $\mathbb{C}\setminus[1,\infty)$.  If $\mathbf{k}$ is admissible, it converges absolutely on the disc $|z| \le 1$ and $\Li_{\mathbf{k}}(1)=\zeta(\mathbf{k})$ holds.

In this paper, we evaluate iterated log-sine integrals by using multiple polylogarithms and multiple zeta values.
For generalized log-sine integrals, such a subject are discussed in \cite[Section 5]{BS}, \cite{DK1}, \cite{DK2} and \cite[Section 2.3]{KS}.
The main theorem in this paper is the following theorem shown in Section \ref{se:evalLs}.
\begin{theorem}\label{th:evalLs}Let \[\mathbf{1}_{n} = (\{1\}^{n}).\]
For $\sigma \in [0,2\pi]$, $\mathbf{k} = (k_{1}, \dots, k_{n}) \in \mathbb{N}^{n}$ and  $\mathbf{l} = (l_{1},\dots,l_{n}) \in (\mathbb{Z}_{\ge 0})^{n}$ satisfying $k_{u} -1-l_{u} \ge 0$ for all $u \in \{1,\dots,n\}$, we have
\begin{align}
&\Ls_{\mathbf{k}}^{\mathbf{l}}(\sigma) \label{eq:main}\\
&=i^{|\mathbf{l}|+n}(-1)^{|\mathbf{k}|+n}\sum_{\mathbf{p}+\mathbf{q}+\mathbf{r}=\mathbf{k}-\mathbf{1}_{n}-\mathbf{l}}\frac{\left(-i\pi\right)^{|\mathbf{p}|}}{2^{|\mathbf{p}|+|\mathbf{q}|}}\binom{\mathbf{k}-\mathbf{1}_{n}-\mathbf{l}}{\mathbf{p},\mathbf{q},\mathbf{r}}F_{\mathbf{q}+\mathbf{l}}^{\mathbf{r}}(\sigma),\nonumber
\end{align}
where the sum is over all $\mathbf{p}=(p_{1},\dots,p_{n})\in(\mathbb{Z}_{\ge0})^{n}$, $\mathbf{q}=(q_{1},\dots,q_{n})\in(\mathbb{Z}_{\ge0})^{n}$ and $\mathbf{r}=(r_{1},\dots,r_{n})\in(\mathbb{Z}_{\ge0})^{n}$ satisfying $\mathbf{p}+\mathbf{q}+\mathbf{r}=\mathbf{k}-\mathbf{1}_{n}-\mathbf{l}$, and
\[\binom{\mathbf{k}-\mathbf{1}_{n}-\mathbf{l}}{\mathbf{p},\mathbf{q},\mathbf{r}} = \prod_{u=1}^{n}\frac{(k_{u}-1-l_{u})!}{p_{u}!q_{u}!r_{u}!}.\]
\end{theorem}
The function $F_{\mathbf{q}+\mathbf{l}}^{\mathbf{r}}(\sigma)$ will be defined in Section \ref{se:defF}. In particular, the function  $F_{\mathbf{q}+\mathbf{l}}^{\mathbf{r}}(\sigma)$ is written as a $\mathbb{Q}(i)$-linear combination of products of a power of $\sigma$, multiple zeta values and a multiple polylogarithm with admissible index at $e^{i\sigma}$. The equation (\ref{eq:main}) makes numerical evaluation of iterated log-sine integrals possible. Another algorithm for numerical evaluation of the generalized log-sine integral was given by Kalmykov and Sheplyakov \cite{KS}.

The real and imaginary parts of multiple polylogarithms at $e^{i\sigma}$ are called as follows.
\begin{definition}[Multiple Clausen function, multiple Glaisher function]\label{de:CG}
The multiple Clausen function $\Cl_{\mathbf{k}}(\sigma)$ and the multiple Glaisher function $\Gl_{\mathbf{k}}(\sigma)$ are defined by 
\begin{align*}
\Cl_{\mathbf{k}}(\sigma)&=\begin{cases}
\Im(\Li_{\mathbf{k}}(e^{i\sigma}))\quad \text{if}\ |\mathbf{k}|:{\rm even},\\
\Re(\Li_{\mathbf{k}}(e^{i\sigma}))\quad \text{if}\ |\mathbf{k}|:{\rm odd},
\end{cases}\\
\Gl_{\mathbf{k}}(\sigma)&=\begin{cases}
\Re(\Li_{\mathbf{k}}(e^{i\sigma}))\quad \text{if}\ |\mathbf{k}|:{\rm even},\\
\Im(\Li_{\mathbf{k}}(e^{i\sigma}))\quad \text{if}\ |\mathbf{k}|:{\rm odd}.
\end{cases}
\end{align*}
\end{definition}
In particular, we call the values $\Cl_{\mathbf{k}}(\pi/3)$ and $\Gl_{\mathbf{k}}(\pi/3)$ multiple Clausen values and multiple Glaisher values, respectively. The real part of the equation (\ref{eq:main}) can be also regarded as a relation among an iterated log-sine integral, multiple zeta values, multiple Clausen functions and Glaisher functions.
Similarly, the imaginary part of the equation (\ref{eq:main}) can be regarded as a relation among multiple zeta values, multiple Clausen functions and Glaisher functions.

In Section \ref{se:evalLsgeneral}, we discuss the iterated log-sine integrals at a general argument not only in the case $\sigma\in[0, 2\pi]$ and give the following theorem.
\begin{theorem}\label{th:Lsgeneral}
For $\sigma \in [0,2\pi]$ and $m\ge0$, the iterated log-sine integral $\Ls_{\mathbf{k}}^{\mathbf{l}}(\pm(2m\pi+\sigma))$ satisfying $k_{u} -1-l_{u} \ge 0$ for all $u \in \{1,\dots,n\}$ can be written as a $\mathbb{Q}(i)$-linear combination of products of a power of $\sigma$, a power of $\pi$, multiple zeta values and multiple polylogarithms with admissible index at $e^{i\sigma}$.
\end{theorem}
In Section \ref{se:Con}, we will introduce and discuss the shifted log-sine integrals. Then, we suggest some conjectures on shifted log-sine integrals, multiple zeta values, multiple Clausen values, multiple Glaisher values and iterated log-sine integrals. Lower weight parts of the conjectures have been checked by numerical evaluation.
\section{Evaluation of iterated log-sine integrals at $\sigma\in[0, 2\pi]$}\label{se:evalLs}
In this section, we evaluate iterated log-sine integrals in terms of multiple polylogarithms and multiple zeta values and give a proof of Theorem \ref{th:evalLs}. 
\subsection{Evaluation of iterated log-sine integrals in terms of ${\rm ILi}_{\mathbf{q}}^{\mathbf{r}}(\sigma)$} For $\mathbf{k}=(k_{1},\dots,k_{n})$,  we define $\mathbf{k}_{-} = (k_{1},\dots,k_{n-1})$.
In order to prove Theorem \ref{th:evalLs}, we first introduce ${\rm ILi}_{\mathbf{q}}^{\mathbf{r}}(\sigma)$ defined as follows:
\begin{definition}
For $\mathbf{q} = (q_{1},\dots,q_{n}) \in (\mathbb{Z}_{\ge 0})^{n}$, $\mathbf{r} = (r_{1},\dots,r_{n}) \in (\mathbb{Z}_{\ge 0})^{n}$ and $\sigma \in [0,2\pi]$, we define
\begin{align*}
{\rm ILi}_{\mathbf{q}}^{\mathbf{r}}(\sigma) &= \int_{0}^{\sigma}\int_{0}^{\theta_{n}}\cdots\int_{0}^{\theta_{2}}
\prod_{u=1}^{n}i(i\theta_{u})^{q_{u}}\left(\Li_{1}(e^{i\theta_{u}})\right)^{r_{u}}\,d\theta_{1}\cdots d\theta_{n}\\
&=\int_{0}^{\sigma}i(i\theta_{n})^{q_{n}}\left(\Li_{1}(e^{i\theta_{n}})\right)^{r_{n}}{\rm ILi}_{\mathbf{q}_{-}}^{\mathbf{r}_{-}}(\theta_{n})\,d\theta_{n},
\end{align*}
where ${\rm ILi}_{\emptyset}^{\emptyset}(\sigma)$ is regarded as $1$.
\end{definition}
Then, we can evaluate iterated log-sine integrals in terms of ${\rm ILi}_{\mathbf{q}}^{\mathbf{r}}(\sigma)$ as follows:
\begin{lemma}\label{le:LstoILi}
For $\sigma \in [0,2\pi]$, $\mathbf{k} = (k_{1}, \dots, k_{n}) \in \mathbb{N}^{n}$ and $\mathbf{l} = (l_{1},\dots,l_{n}) \in (\mathbb{Z}_{\ge 0})^{n}$ satisfying $k_{u} -1-l_{u} \ge 0$ for all $u \in \{1,\dots,n\}$, we have
\begin{align*}
&\Ls_{\mathbf{k}}^{\mathbf{l}}(\sigma)=i^{|\mathbf{l}|+n}(-1)^{|\mathbf{k}|+n}\sum_{\mathbf{p}+\mathbf{q}+\mathbf{r}=\mathbf{k}-\mathbf{1}_{n}-\mathbf{l}}\frac{\left(-i\pi\right)^{|\mathbf{p}|}}{2^{|\mathbf{p}|+|\mathbf{q}|}}\binom{\mathbf{k}-\mathbf{1}_{n}-\mathbf{l}}{\mathbf{p},\mathbf{q},\mathbf{r}}{\rm ILi}_{\mathbf{q}+\mathbf{l}}^{\mathbf{r}}(\sigma).
\end{align*}
\end{lemma}
\begin{proof}
Noting $\Li_{1}(e^{i\theta}) = -\log(1-e^{i\theta}) = -A(\theta)-i\frac{\theta-\pi}{2}$ $(0<\theta<2\pi)$, we have
\begin{align*}
&(-i)^{|\mathbf{l}|+n}(-1)^{|\mathbf{k}|+n}\Ls_{\mathbf{k}}^{\mathbf{l}}(\sigma)\\
&=(-i)^{|\mathbf{l}|+n}(-1)^{|\mathbf{k}|}\int_{0<\theta_{1}<\dots<\theta_{n}<\sigma}\prod_{u=1}^{n}\theta_u^{l_u}A^{k_u-1-l_u}(\theta_{u})\,d\theta_u\\
&=\int_{0<\theta_{1}<\dots<\theta_{n}<\sigma}\prod_{u=1}^{n}i
\left(i\theta_u\right)^{l_u}\left(-A(\theta_{u})\right)^{k_u-1-l_u}\,d\theta_u\\
&=\int_{0<\theta_{1}<\dots<\theta_{n}<\sigma}\prod_{u=1}^{n}i
\left(i\theta_u\right)^{l_u}\left(\Li_{1}(e^{i\theta_{u}})+\frac{i\theta_{u}}{2}-\frac{i\pi}{2}\right)^{k_u-1-l_u}\,d\theta_u\\
&=\int_{0<\theta_{1}<\dots<\theta_{n}<\sigma}\prod_{u=1}^{n}i
\left(i\theta_u\right)^{l_u}\sum_{p_{u}+q_{u}+r_{u}=
k_{u}-1+l_{u}}\frac{(k_{u}-1-l_{u})!}{p_{u}!q_{u}!r_{u}!}\\
&\quad\times\left(-\frac{i\pi}{2}\right)^{p_{u}}\left(\frac{i\theta_{u}}{2}\right)^{q_{u}}\left(\Li_{1}(e^{i\theta_{u}})\right)^{r_{u}}\,d\theta_u\\
&=\sum_{\mathbf{p}+\mathbf{q}+\mathbf{r}=\mathbf{k}-\mathbf{1}_{n}-\mathbf{l}}\frac{\left(-i\pi\right)^{|\mathbf{p}|}}{2^{|\mathbf{p}|+|\mathbf{q}|}}\left(\prod_{u=1}^{n}\frac{(k_{u}-1-l_{u})!}{p_{u}!q_{u}!r_{u}!}\right)\\
&\quad\times\int_{0<\theta_{1}<\dots<\theta_{n}<\sigma}\prod_{u=1}^{n}i
\left(i\theta_{u}\right)^{q_{u}+l_{u}}\left(\Li_{1}(e^{i\theta_{u}})\right)^{r_{u}}\,d\theta_u.
\end{align*}
\end{proof}
In order to prove Theorem \ref{th:evalLs},
we only need to prove ${\rm ILi}_{\mathbf{q}}^{\mathbf{r}}(\sigma)=F_{\mathbf{q}}^{\mathbf{r}}(\sigma)$.
\subsection{Evaluation of ${\rm ILi}_{\mathbf{q}}^{\mathbf{r}}(\sigma)$}\label{se:defF}
Let $\mathbf{q},\,\mathbf{r},\,\mathbf{j}\in(\mathbb{Z}_{\ge 0})^{n}$. We define the rational number $B_{\mathbf{q}}$ by
\[B_{\mathbf{q}} = \frac{1}{|\mathbf{q}|+n}B_{\mathbf{q}_{-}} \quad \text{and} \quad B_{\emptyset}=1,\]
and the rational number $C_{\mathbf{q}}^{\mathbf{j}}$ by
\[C_{\mathbf{q}}^{\mathbf{j}} = (-1)^{j_{n}}\frac{(|\mathbf{q}|-|\mathbf{j}_{-}|)!}{(|\mathbf{q}|-|\mathbf{j}|)!} C_{\mathbf{q}_{-}}^{\mathbf{j}_{-}} \quad \text{and} \quad C_{\emptyset}^{\emptyset}=1.\]
We write $\mathbf{j} \preceq \mathbf{q}$ when $\mathbf{j}$ and $\mathbf{q}$ satisfy
\begin{align*}
j_{1}&\le q_{1}\\
j_{1}+j_{2}&\le q_{1}+q_{2}\\
&\ \ \vdots\\
j_{1}+j_{2}+\dots+j_{n} &\le q_{1}+q_{2}+\dots+q_{n},
\end{align*}
and understand $\emptyset\preceq\emptyset$.
We define the noncommutative polynomial ring $\mathfrak{H} = \mathbb{Q}\langle e_{0}, e_{1}\rangle$ and its subspace $\mathfrak{H}^{0} = \mathbb{Q}+e_{1}\mathfrak{H}e_{0}$ and $\mathfrak{H}^{1} = \mathbb{Q}+e_{1}\mathfrak{H}$, and the element $w_{\mathbf{j}}^{\mathbf{r}}$ in $\mathfrak{H}$ by
\[w_{\mathbf{j}}^{\mathbf{r}} =  (w_{\mathbf{j}_{-}}^{\mathbf{r}_{-}} \shuffle e_{1}^{\shuffle r_{n}})e_{0}^{1+j_{n}} \quad {\rm and} \quad w_{\emptyset}^{\emptyset}=1,\]
where the shuffle product $\shuffle$ is recursively defined by
\begin{align*}
w\shuffle1&=w,\ \,\,1\shuffle w=w\qquad\qquad\qquad\qquad(w:\text{word}),\\
u_{1}w_{1}\shuffle u_{2}w_{2}&=u_{1}(w_{1}\shuffle u_{2}w_{2})+u_{2}(u_{1}w_{1}\shuffle w_{2})\ (u_{1},u_{2} \in \{e_{0},e_{1}\}, w_{1},w_{2}:\text{words})
\end{align*}
with $\mathbb{Q}$-bilinearity, and $e_{1}^{\shuffle r_{n}}$ represents
\[\underbrace{e_{1}\shuffle \cdots \shuffle e_{1}}_{r_{n}}.\] Note that $w_{\mathbf{j}}^{\mathbf{r}} \in \mathfrak{H}^{0}$ when $r_{1}\ge1$.
We define $\mathbb{Q}$-linear maps $Z\colon \mathfrak{H}^{0} \to \mathbb{R}$ and $L(\,\cdot\,;z )\colon \mathfrak{H}^{1}\to \mathbb{C}$ by
\begin{align*}
Z(e_{1} e_{0}^{k_{1}-1}\cdots e_{1}e_{0}^{k_{n}-1}) &= \zeta(k_{1},\dots,k_{n}),
\\
L(e_{1} e_{0}^{k_{1}-1}\cdots e_{1}e_{0}^{k_{n}-1};z) &= \Li_{k_{1},\dots,k_{n}}(z),
\end{align*}
and
\begin{align*}
Z(1) &=\zeta(\emptyset)=1,
\\
L(1;z) &= \Li_{\emptyset}(z)=1,
\end{align*}
respectively. Note that $L(w_{1};z)L(w_{2};z)=L(w_{1}\shuffle w_{2};z)$ for any $w_{1}, w_{2} \in \mathfrak{H}^{1}$.

For a pair of indices $\mathbf{q}, \mathbf{r}\in (\mathbb{Z}_{\ge0})^{n}$, we also write $\mathbf{r}=(\{0\}^{n'},r''_{1},\dots,r''_{n''})=(\mathbf{r}',\mathbf{r}'') \in (\mathbb{Z}_{\ge 0})^{n'+n''}$ with $r''_{1} \ge 1$, and $\mathbf{q}=(q'_{1},\dots,q'_{n'},q''_{1},\dots,q''_{n''})=(\mathbf{q}',\mathbf{q}'')\in (\mathbb{Z}_{\ge 0})^{n'+n''}$. Note that $q''_{n''} = q_{n}$, $r''_{n''} = r_{n}$ and $n=n'+n''$. We define $\overline{\mathbf{q}} = (|\mathbf{q}'|+n'+q_{1}'',q_{2}'',\dots,q_{n''}'')$ when $\mathbf{q}''\neq\emptyset$, and $\overline{\mathbf{q}}=\emptyset$ when $\mathbf{q}''=\emptyset$. For example, if $\mathbf{q}=(1,2,3,4)$ and $\mathbf{r}=(0,0,1,2)$ then
$\mathbf{q}^{\prime}=(1,2)$, $\mathbf{q}^{\prime\prime}=(3,4)$, $\mathbf{r}^{\prime}=(0,0)$, $\mathbf{r}^{\prime\prime}=(1,2)$ and $\overline{\mathbf{q}}=(8,4)$.

\begin{definition}\label{def:f}
For $\sigma \in [0,2\pi]$, $\mathbf{q} \in (\mathbb{Z}_{\ge 0})^{n}$ and $\mathbf{r} \in (\mathbb{Z}_{\ge 0})^{n}$, we define
\begin{align*}
f_{\mathbf{q}}^{\mathbf{r}}(\sigma) = 
B_{\mathbf{q}'}\sum_{\mathbf{j} \preceq \overline{\mathbf{q}}}C_{\overline{\mathbf{q}}}^{\mathbf{j}}(i\sigma)^{|\mathbf{q}|+n'-|\mathbf{j}|}L(w_{\mathbf{j}}^{\mathbf{r}''};e^{i\sigma}),
\end{align*}
where the sum is over $\mathbf{j}=(j_{1},\dots,j_{n''})\in(\mathbb{Z}_{\ge 0})^{n''}$ satisfying $\mathbf{j} \preceq \overline{\mathbf{q}}$.
\end{definition}
The function $f_{\mathbf{q}}^{\mathbf{r}}(\sigma)$ can also be written as follows:
\begin{align}
f_{\mathbf{q}}^{\mathbf{r}}(\sigma) = \label{eq:f(sigma)}
\begin{cases}
1 &{\rm if}\ \mathbf{r}=\emptyset,\\
B_{\mathbf{q}}(i\sigma)^{|\mathbf{q}|+n}
&{\rm if}\ \mathbf{r}=(\{0\}^{n}), n\ge 1,\\
B_{\mathbf{q}'}\sum_{\mathbf{j} \preceq \overline{\mathbf{q}}}C_{\overline{\mathbf{q}}}^{\mathbf{j}}(i\sigma)^{|\overline{\mathbf{q}}|-|\mathbf{j}|}L(w_{\mathbf{j}}^{\mathbf{r}''};e^{i\sigma})&{\rm otherwise}.
\end{cases}
\end{align}
\begin{remark}\label{re:f0}We have
\begin{align}
f_{\mathbf{q}}^{\mathbf{r}}(0) = \label{eq:f(0)}
\begin{cases}
1 &{\rm if}\ \mathbf{r}=\emptyset,\\
0&{\rm if}\ \mathbf{r}=(\{0\}^{n}), n\ge 1,\\
\begin{gathered}
B_{\mathbf{q}'}\sum_{\mathbf{j}_{-} \preceq \overline{\mathbf{q}}_{-}}C_{\overline{\mathbf{q}}_{-}}^{\mathbf{j}_{-}}(-1)^{|\overline{\mathbf{q}}|-|\mathbf{j}_{-}|}(|\overline{\mathbf{q}}|-|\mathbf{j}_{-}|)!\\
\times Z\left((w_{\mathbf{j}_{-}}^{\mathbf{r}''_{-}} \shuffle e_{1}^{\shuffle r_{n}})e_{0}^{1+|\overline{\mathbf{q}}|-|\mathbf{j}_{-}|}\right)
\end{gathered}&{\rm otherwise}.
\end{cases}
\end{align}
\end{remark}
\begin{proof}
We prove only the case ``otherwise''. In this case, we obtain
\begin{align}
f_{\mathbf{q}}^{\mathbf{r}}(\sigma) &= B_{\mathbf{q}'}\sum_{\mathbf{j} \preceq \overline{\mathbf{q}}}C_{\overline{\mathbf{q}}}^{\mathbf{j}}(i\sigma)^{|\overline{\mathbf{q}}|-|\mathbf{j}|}L\left(w_{\mathbf{j}}^{\mathbf{r}''};e^{i\sigma}\right)\label{eq:ind}\\
&=B_{\mathbf{q}'}\sum_{\mathbf{j}_{-} \preceq \overline{\mathbf{q}}_{-}}C_{\overline{\mathbf{q}}_{-}}^{\mathbf{j}_{-}}\sum_{j_{n''}=0}^{|\overline{\mathbf{q}}|-|\mathbf{j}_{-}|}(-1)^{j_{n''}}\frac{(|\overline{\mathbf{q}}|-|\mathbf{j}_{-}|)!}{(|\overline{\mathbf{q}}|-|\mathbf{j}|)!}\nonumber\\
&\quad\times(i\sigma)^{|\overline{\mathbf{q}}|-|\mathbf{j}|}L\left((w_{\mathbf{j}_{-}}^{\mathbf{r}''_{-}} \shuffle e_{1}^{\shuffle r_{n}})e_{0}^{1+j_{n''}};e^{i\sigma}\right).\nonumber
\end{align}
By putting $\sigma=0$, we find that the factor $(i\sigma)^{|\overline{\mathbf{q}}|-|\mathbf{j}|}$ is equal to $0$ when $j_{n''}\neq|\overline{\mathbf{q}}|-|\mathbf{j}_{-}|$, and $1$ when $j_{n''}=|\overline{\mathbf{q}}|-|\mathbf{j}_{-}|$. Therefore, we obtain the desired formula.
\end{proof}
\begin{proposition}\label{pr:diff}
For $\mathbf{q}, \mathbf{r}\in(\mathbb{Z}_{\ge0})^{n}$ with $n\ge1$, we have
\[\frac{d}{d\sigma}f_{\mathbf{q}}^{\mathbf{r}}(\sigma) = i(i\sigma)^{q_{n}}\left(\Li_{1}(e^{i\sigma})\right)^{r_{n}}f_{\mathbf{q}_{-}}^{\mathbf{r}_{-}}(\sigma).\]
\end{proposition}

\begin{proof}
In the case $\mathbf{r}=(\{0\}^{n})$, $n\ge 1$, we have
\begin{align*}
\frac{d}{d\sigma}f_{\mathbf{q}}^{\mathbf{r}}(\sigma)&=\frac{d}{d\sigma}B_{\mathbf{q}}(i\sigma)^{|\mathbf{q}|+n}\\
&=B_{\mathbf{q}_{-}}i(i\sigma)^{|\mathbf{q}|+n-1}\\
&=i(i\sigma)^{q_{n}}B_{\mathbf{q}_{-}}(i\sigma)^{|\mathbf{q}_{-}|+n-1}\\
&=i(i\sigma)^{q_{n}}f_{\mathbf{q}_{-}}^{\mathbf{r}_{-}}(\sigma).
\end{align*}
In the case ``otherwise'', that is $n''\ge1$, by equation (\ref{eq:ind}) we obtain
\begin{align*}
\frac{d}{d\sigma}f_{\mathbf{q}}^{\mathbf{r}}(\sigma) 
&=B_{\mathbf{q}'}\sum_{\mathbf{j}_{-} \preceq \overline{\mathbf{q}}_{-}}C_{\overline{\mathbf{q}}_{-}}^{\mathbf{j}_{-}}\sum_{j_{n''}=0}^{|\overline{\mathbf{q}}|-|\mathbf{j}_{-}|}(-1)^{j_{n''}}\frac{(|\overline{\mathbf{q}}|-|\mathbf{j}_{-}|)!}{(|\overline{\mathbf{q}}|-|\mathbf{j}|)!}\\
&\quad\times\frac{d}{d\sigma}(i\sigma)^{|\overline{\mathbf{q}}|-|\mathbf{j}|}L\left((w_{\mathbf{j}_{-}}^{\mathbf{r}''_{-}} \shuffle e_{1}^{\shuffle r_{n}})e_{0}^{1+j_{n''}};e^{i\sigma}\right).
\end{align*}
Here,
\begin{align*}
&\sum_{j_{n''}=0}^{|\overline{\mathbf{q}}|-|\mathbf{j}_{-}|}(-1)^{j_{n''}}\frac{(|\overline{\mathbf{q}}|-|\mathbf{j}_{-}|)!}{(|\overline{\mathbf{q}}|-|\mathbf{j}|)!}\frac{d}{d\sigma} (i\sigma)^{|\overline{\mathbf{q}}|-|\mathbf{j}|}L\left((w_{\mathbf{j}_{-}}^{\mathbf{r}''_{-}} \shuffle e_{1}^{\shuffle r_{n}})e_{0}^{1+j_{n''}};e^{i\sigma}\right)\\
&=\sum_{j_{n''}=0}^{|\overline{\mathbf{q}}|-|\mathbf{j}_{-}|-1}(-1)^{j_{n''}}\frac{(|\overline{\mathbf{q}}|-|\mathbf{j}_{-}|)!}{(|\overline{\mathbf{q}}|-|\mathbf{j}|)!}\frac{d}{d\sigma}(i\sigma)^{|\overline{\mathbf{q}}|-|\mathbf{j}|}L\left((w_{\mathbf{j}_{-}}^{\mathbf{r}''_{-}} \shuffle e_{1}^{\shuffle r_{n}})e_{0}^{1+j_{n''}};e^{i\sigma}\right)\\
&\quad+(-1)^{|\overline{\mathbf{q}}|-|\mathbf{j}_{-}|}(|\overline{\mathbf{q}}|-|\mathbf{j}_{-}|)!\frac{d}{d\sigma}L\left((w_{\mathbf{j}_{-}}^{\mathbf{r}''_{-}} \shuffle e_{1}^{\shuffle r_{n}})e_{0}^{1+|\overline{\mathbf{q}}|-|\mathbf{j}_{-}|};e^{i\sigma}\right)\nonumber\\
&=\sum_{j_{n''}=0}^{|\mathbf{q}|-|\mathbf{j}_{-}|-1}\left((-1)^{j_{n''}}\frac{(|\overline{\mathbf{q}}|-|\mathbf{j}_{-}|)!}{(|\overline{\mathbf{q}}|-|\mathbf{j}|-1)!} i(i\sigma)^{|\overline{\mathbf{q}}|-|\mathbf{j}|-1}L\left((w_{\mathbf{j}_{-}}^{\mathbf{r}''_{-}} \shuffle e_{1}^{\shuffle r_{n}})e_{0}^{1+j_{n''}};e^{i\sigma}\right)\right.\\
&\quad+\left.(-1)^{j_{n''}}\frac{(|\overline{\mathbf{q}}|-|\mathbf{j}_{-}|)!}{(|\overline{\mathbf{q}}|-|\mathbf{j}|)!} i(i\sigma)^{|\overline{\mathbf{q}}|-|\mathbf{j}|}L\left((w_{\mathbf{j}_{-}}^{\mathbf{r}''_{-}} \shuffle e_{1}^{\shuffle r_{n}})e_{0}^{j_{n''}};e^{i\sigma}\right)\right)\\
&\quad+(-1)^{|\overline{\mathbf{q}}|-|\mathbf{j}_{-}|}(|\overline{\mathbf{q}}|-|\mathbf{j}_{-}|)!i L\left((w_{\mathbf{j}_{-}}^{\mathbf{r}''_{-}} \shuffle e_{1}^{\shuffle r_{n}})e_{0}^{|\overline{\mathbf{q}}|-|\mathbf{j}_{-}|};e^{i\sigma}\right)\\
&=i(i\sigma)^{|\overline{\mathbf{q}}|-|\mathbf{j}_{-}|}L\left(w_{\mathbf{j}_{-}}^{\mathbf{r}''_{-}} \shuffle e_{1}^{\shuffle r_{n}};e^{i\sigma}\right).
\end{align*}
Therefore, we obtain
\begin{align*}
\frac{d}{d\sigma}f_{\mathbf{q}}^{\mathbf{r}}(\sigma) 
&=B_{\mathbf{q}'}\sum_{\mathbf{j}_{-} \preceq \overline{\mathbf{q}}_{-}}C_{\overline{\mathbf{q}}_{-}}^{\mathbf{j}_{-}}i(i\sigma)^{|\overline{\mathbf{q}}|-|\mathbf{j}_{-}|}L\left(w_{\mathbf{j}_{-}}^{\mathbf{r}''_{-}} \shuffle e_{1}^{\shuffle r_{n}};e^{i\sigma}\right).
\end{align*}
If $n'' \ge 2$, then we have 
\begin{align*}
\frac{d}{d\sigma}f_{\mathbf{q}}^{\mathbf{r}}(\sigma) 
&=i(i\sigma)^{q_{n}}\left(\Li_{1}(e^{i\sigma})\right)^{r_{n}}B_{\mathbf{q}'}\sum_{\mathbf{j}_{-} \preceq \overline{\mathbf{q}}_{-}}C_{\overline{\mathbf{q}}_{-}}^{\mathbf{j}_{-}}(i\sigma)^{|\overline{\mathbf{q}}_{-}|-|\mathbf{j}_{-}|}L\left(w_{\mathbf{j}_{-}}^{\mathbf{r}''_{-}};e^{i\sigma}\right)\\
&=i(i\sigma)^{q_{n}}\left(\Li_{1}(e^{i\sigma})\right)^{r_{n}}f_{\mathbf{q}_{-}}^{\mathbf{r}_{-}}(\sigma).
\end{align*}
In the case $n'' = 1$, since $\overline{\mathbf{q}}=(|\mathbf{q}'|+n'+q_{n})$, $\mathbf{q}'=\mathbf{q}_{-}$ and $\mathbf{r}'=\mathbf{r}_{-}$, we have
\begin{align*}
\frac{d}{d\sigma}f_{\mathbf{q}}^{\mathbf{r}}(\sigma) 
&=B_{\mathbf{q}'}i(i\sigma)^{|\mathbf{q}'|+n'+q_{n}}\left(\Li_{1}(e^{i\sigma})\right)^{r_{n}}\\
&=i(i\sigma)^{q_{n}}\left(\Li_{1}(e^{i\sigma})\right)^{r_{n}}B_{\mathbf{q}'}(i\sigma)^{|\mathbf{q}'|+n'}\\
&=i(i\sigma)^{q_{n}}\left(\Li_{1}(e^{i\sigma})\right)^{r_{n}}f_{\mathbf{q}_{-}}^{\mathbf{r}_{-}}(\sigma).\\
\end{align*}
Therefore, in all cases, this proposition holds.
\end{proof}
\begin{definition}\label{def:F}
For $\sigma \in [0,2\pi]$, $\mathbf{q} \in (\mathbb{Z}_{\ge 0})^{n}$ and $\mathbf{r} \in (\mathbb{Z}_{\ge 0})^{n}$, we define
\begin{align*}F_{\mathbf{q}}^{\mathbf{r}}(\sigma) =\sum_{\substack{(\mathbf{q}^{(1)},\dots,\mathbf{q}^{(h)})=\mathbf{q} \\ (\mathbf{r}^{(1)},\dots,\mathbf{r}^{(h)})=\mathbf{r}\\ {\rm dep}(\mathbf{q}^{(j)})={\rm dep}(\mathbf{r}^{(j)}) \ge 1 \\ (1 \le j \le h)}}(-1)^{h-1}\left(\prod_{j=1}^{h-1}f_{\mathbf{q}^{(j)}}^{\mathbf{r}^{(j)}}(0)\right)\left(f_{\mathbf{q}^{(h)}}^{\mathbf{r}^{(h)}}(\sigma)-f_{\mathbf{q}^{(h)}}^{\mathbf{r}^{(h)}}(0)\right),
\end{align*}
where the sum is over all partitions of $\mathbf{q}$ and $\mathbf{r}$, for example
\begin{align*}
&F_{(q_{1},q_{2},q_{3})}^{(r_{1},r_{2},r_{3})}(\sigma) \\
&=f_{(q_{1})}^{(r_{1})}(0)f_{(q_{2})}^{(r_{2})}(0)\left(f_{(q_{3})}^{(r_{3})}(\sigma)-f_{(q_{3})}^{(r_{3})}(0)\right)\\
&\quad-f_{(q_{1},q_{2})}^{(r_{1},r_{2})}(0)\left(f_{(q_{3})}^{(r_{3})}(\sigma)-f_{(q_{3})}^{(r_{3})}(0)\right)-f_{(q_{1})}^{(r_{1})}(0)\left(f_{(q_{2},q_{3})}^{(r_{2},r_{3})}(\sigma)-f_{(q_{2},q_{3})}^{(r_{2},r_{3})}(0)\right)\\
&\quad+\left(f_{(q_{1},q_{2},q_{3})}^{(r_{1},r_{2},r_{3})}(\sigma)-f_{(q_{1},q_{2},q_{3})}^{(r_{1},r_{2},r_{3})}(0)\right).
\end{align*}
We understand $F_{\emptyset}^{\emptyset}(\sigma)=1$.
\end{definition}
\begin{remark}
By (\ref{eq:f(sigma)}) and (\ref{eq:f(0)}), we can see that the function  $F_{\mathbf{q}+\mathbf{l}}^{\mathbf{r}}(\sigma)$ is written as a $\mathbb{Q}(i)$-linear combination of products of a power of $\sigma$, multiple zeta values and a multiple polylogarithm with admissible index at $e^{i\sigma}$. 
\end{remark}
\begin{proposition}\label{prop:difF}
For $\mathbf{q}, \mathbf{r}\in(\mathbb{Z}_{\ge0})^{n}$ with $n\ge1$, we have
\[\frac{d}{d\sigma}F_{\mathbf{q}}^{\mathbf{r}}(\sigma) = i(i\sigma)^{q_{n}}\left(\Li_{1}(e^{i\sigma})\right)^{r_{n}}F_{\mathbf{q}_{-}}^{\mathbf{r}_{-}}(\sigma).\]
\end{proposition}
\begin{proof}
In the case $n=1$, by Proposition \ref{pr:diff}, we have
\[\frac{d}{d\sigma}F_{(q_{1})}^{(r_{1})}(\sigma) =\frac{d}{d\sigma}\left(f_{(q_{1})}^{(r_{1})}(\sigma)-f_{(q_{1})}^{(r_{1})}(0)\right)= i(i\sigma)^{q_{1}}\left(\Li_{1}(e^{i\sigma})\right)^{r_{1}}.\]
In the case $n \ge 2$, we have
\begin{align*}
&\frac{d}{d\sigma}F_{\mathbf{q}}^{\mathbf{r}}(\sigma) \\
&=i(i\sigma)^{q_{n}}\left(\Li_{1}(e^{i\sigma})\right)^{r_{n}}\sum_{\substack{(\mathbf{q}^{(1)},\dots,\mathbf{q}^{(h)})=\mathbf{q} \\ (\mathbf{r}^{(1)},\dots,\mathbf{r}^{(h)})=\mathbf{r} \\{\rm dep}(\mathbf{q}^{(j)})={\rm dep}(\mathbf{r}^{(j)})\ge1 \\ (1 \le j \le h)}}(-1)^{h-1}\left(\prod_{j=1}^{h-1}f_{\mathbf{q}^{(j)}}^{\mathbf{r}^{(j)}}(0)\right)f_{(\mathbf{q}^{(h)})_{-}}^{(\mathbf{r}^{(h)})_{-}}(\sigma)\\
&=i(i\sigma)^{q_{n}}\left(\Li_{1}(e^{i\sigma})\right)^{r_{n}}\Biggl(\sum_{\substack{(\mathbf{q}^{(1)},\dots,\mathbf{q}^{(h)})=\mathbf{q} \\ (\mathbf{r}^{(1)},\dots,\mathbf{r}^{(h)})=\mathbf{r}\\{\rm dep}(\mathbf{q}^{(j)})={\rm dep}(\mathbf{r}^{(j)})\ge1 \\ (1 \le j \le h)\\ {\rm dep}(\mathbf{q}^{(h)})={\rm dep}(\mathbf{r}^{(h)}) \ge 2 }}+\sum_{\substack{(\mathbf{q}^{(1)},\dots,\mathbf{q}^{(h)})=\mathbf{q} \\ (\mathbf{r}^{(1)},\dots,\mathbf{r}^{(h)})=\mathbf{r}\\ {\rm dep}(\mathbf{q}^{(j)})={\rm dep}(\mathbf{r}^{(j)})\ge1 \\ (1 \le j \le h)\\ {\rm dep}(\mathbf{q}^{(h)})={\rm dep}(\mathbf{r}^{(h)}) = 1}}\Biggr)\\
&\qquad(-1)^{h-1}\left(\prod_{j=1}^{h-1}f_{\mathbf{q}^{(j)}}^{\mathbf{r}^{(j)}}(0)\right)f_{(\mathbf{q}^{(h)})_{-}}^{(\mathbf{r}^{(h)})_{-}}(\sigma)\\
&=i(i\sigma)^{q_{n}}\left(\Li_{1}(e^{i\sigma})\right)^{r_{n}}\sum_{\substack{(\mathbf{q}^{(1)},\dots,\mathbf{q}^{(h)})=\mathbf{q}_{-} \\ (\mathbf{r}^{(1)},\dots,\mathbf{r}^{(h)})=\mathbf{r}_{-}\\{\rm dep}(\mathbf{q}^{(j)})={\rm dep}(\mathbf{r}^{(j)}) \ge 1 \\ (1 \le j \le h)}}(-1)^{h-1}\left(\prod_{j=1}^{h-1}f_{\mathbf{q}^{(j)}}^{\mathbf{r}^{(j)}}(0)\right)f_{\mathbf{q}^{(h)}}^{\mathbf{r}^{(h)}}(\sigma)\\
&\quad+i(i\sigma)^{q_{n}}\left(\Li_{1}(e^{i\sigma})\right)^{r_{n}}\sum_{\substack{(\mathbf{q}^{(1)},\dots,\mathbf{q}^{(h-1)})=\mathbf{q}_{-} \\ (\mathbf{r}^{(1)},\dots,\mathbf{r}^{(h-1)})=\mathbf{r}_{-}\\{\rm dep}(\mathbf{q}^{(j)})={\rm dep}(\mathbf{r}^{(j)})\ge1 \\ (1 \le j \le h-1)}}(-1)^{h-1}\left(\prod_{j=1}^{h-1}f_{\mathbf{q}^{(j)}}^{\mathbf{r}^{(j)}}(0)\right).
\end{align*}
Because
\begin{align*}
&\sum_{\substack{(\mathbf{q}^{(1)},\dots,\mathbf{q}^{(h-1)})=\mathbf{q}_{-} \\ (\mathbf{r}^{(1)},\dots,\mathbf{r}^{(h-1)})=\mathbf{r}_{-}\\{\rm dep}(\mathbf{q}^{(j)})={\rm dep}(\mathbf{r}^{(j)})\ge1 \\ (1 \le j \le h-1)}}(-1)^{h-1}\left(\prod_{j=1}^{h-1}f_{\mathbf{q}^{(j)}}^{\mathbf{r}^{(j)}}(0)\right)\\&=-\sum_{\substack{(\mathbf{q}^{(1)},\dots,\mathbf{q}^{(h)})=\mathbf{q}_{-} \\ (\mathbf{r}^{(1)},\dots,\mathbf{r}^{(h)})=\mathbf{r}_{-}\\{\rm dep}(\mathbf{q}^{(j)})={\rm dep}(\mathbf{r}^{(j)}) \ge1 \\ (1 \le j \le h)}}(-1)^{h-1}\left(\prod_{j=1}^{h-1}f_{\mathbf{q}^{(j)}}^{\mathbf{r}^{(j)}}(0)\right)f_{\mathbf{q}^{(h)}}^{\mathbf{r}^{(h)}}(0),
\end{align*}
we obtain the desired formula.
\end{proof}
\begin{proposition}\label{pr:ILitoF}
We have
\[{\rm ILi}_{\mathbf{q}}^{\mathbf{r}}(\sigma)= F_{\mathbf{q}}^{\mathbf{r}}(\sigma).\]
\end{proposition}
Therefore, by Lemma \ref{le:LstoILi} and Proposition \ref{pr:ILitoF}, Theorem \ref{th:evalLs} holds.
\begin{proof}
We prove this proposition by induction on $n$. In the case $n=1$, by Proposition \ref{prop:difF}, we have
\[\frac{d}{d\sigma}F_{(q_{1})}^{(r_{1})}(\sigma) = i(i\sigma)^{q_{1}}\left(\Li_{1}(e^{i\sigma})\right)^{r_{1}}.\]
Therefore, we obtain
\begin{align*}
{\rm ILi}_{(q_{1})}^{(r_{1})}(\sigma) &= \int_{0}^{\sigma}i(i\theta_{1})^{q_{1}}\left(\Li_{1}(e^{i\theta_{1}})\right)^{r_{1}}\,d\theta_{1}\\
&=F_{(q_{1})}^{(r_{1})}(\sigma)-F_{(q_{1})}^{(r_{1})}(0)\\
&=F_{(q_{1})}^{(r_{1})}(\sigma).
\end{align*}
We prove ${\rm ILi}_{\mathbf{q}}^{\mathbf{r}}(\sigma)= F_{\mathbf{q}}^{\mathbf{r}}(\sigma)$ with the assumption ${\rm ILi}_{\mathbf{q}_{-}}^{\mathbf{r}_{-}}(\sigma)= F_{\mathbf{q}_{-}}^{\mathbf{r}_{-}}(\sigma)$.
By Proposition \ref{prop:difF}
\[\frac{d}{d\sigma}F_{\mathbf{q}}^{\mathbf{r}}(\sigma) = i(i\sigma)^{q_{n}}\left(\Li_{1}(e^{i\sigma})\right)^{r_{n}}{\rm ILi}_{\mathbf{q}_{-}}^{\mathbf{r}_{-}}(\sigma),\]
and hence, we obtain 
\begin{align*}
{\rm ILi}_{\mathbf{q}}^{\mathbf{r}}(\sigma) &= \int_{0}^{\sigma}i(i\theta_{n})^{q_{n}}\left(\Li_{1}(e^{i\theta_{n}})\right)^{r_{n}}{\rm ILi}_{\mathbf{q}_{-}}^{\mathbf{r}_{-}}(\theta_{n})\,d\theta_{n}\\
&=F_{\mathbf{q}}^{\mathbf{r}}(\sigma) - F_{\mathbf{q}}^{\mathbf{r}}(0)\\
&=F_{\mathbf{q}}^{\mathbf{r}}(\sigma).
\end{align*}
\end{proof}
\subsection{Examples}
We give some examples of Theorem \ref{th:evalLs}. For $\sigma \in [0, 2\pi]$, we obtain the following evaluations:
\begin{align*}
&{\rm Ls}_{1}^{(0)}(\sigma)=-\sigma,\\
&{\rm Ls}_{2}^{(0)}(\sigma)=-\frac{1}{2} i  \pi \sigma + \frac{1}{4} i  \sigma^{2} + i  \zeta(2) - i  {\rm Li}_{2}(e^{i  \sigma}),\\
&{\rm Ls}_{2}^{(1)}(\sigma)=-\frac{1}{2}  \sigma^{2},\\
&{\rm Ls}_{1 , 1}^{(0 , 0)}(\sigma)=\frac{1}{2}  \sigma^{2},\\
&{\rm Ls}_{3}^{(0)}(\sigma)=\frac{1}{4}  \pi^{2} \sigma - \frac{1}{4}  \pi \sigma^{2} + \frac{1}{12}  \sigma^{3} - \pi \zeta(2) + \pi {\rm Li}_{2}(e^{i  \sigma}) - \sigma {\rm Li}_{2}(e^{i  \sigma})\\
&\qquad\qquad\quad + i  \zeta(3) - i  {\rm Li}_{3}(e^{i  \sigma}) - 2 i  \zeta(1 , 2) + 2 i  {\rm Li}_{1 , 2}(e^{i  \sigma}),\\
&{\rm Ls}_{3}^{(1)}(\sigma)=-\frac{1}{4} i  \pi \sigma^{2} + \frac{1}{6} i  \sigma^{3} - i  \sigma {\rm Li}_{2}(e^{i  \sigma}) - \zeta(3) + {\rm Li}_{3}(e^{i  \sigma}),\\
&{\rm Ls}_{3}^{(2)}(\sigma)=-\frac{1}{3}  \sigma^{3},\\
&{\rm Ls}_{1 , 2}^{(0 , 0)}(\sigma)=\frac{1}{4} i  \pi \sigma^{2} - \frac{1}{6} i  \sigma^{3} + i  \sigma {\rm Li}_{2}(e^{i  \sigma}) + \zeta(3) - {\rm Li}_{3}(e^{i  \sigma}),\\
&{\rm Ls}_{1 , 2}^{(0 , 1)}(\sigma)=\frac{1}{3}  \sigma^{3},\\
&{\rm Ls}_{2 , 1}^{(0 , 0)}(\sigma)=\frac{1}{4} i  \pi \sigma^{2} - \frac{1}{12} i  \sigma^{3} - i  \sigma \zeta(2) - \zeta(3) + {\rm Li}_{3}(e^{i  \sigma}),\\
&{\rm Ls}_{2 , 1}^{(1 , 0)}(\sigma)=\frac{1}{6}  \sigma^{3},\\
&{\rm Ls}_{1 , 1 , 1}^{(0 , 0 , 0)}(\sigma)=-\frac{1}{6}  \sigma^{3}.
\end{align*}
We note that iterated log-sine integrals satisfying $k_{u} -1 - l_{u}=0$ for all $u \in \{1,\dots,n\}$ such as ${\rm Ls}_{1}^{(0)}(\sigma)$, ${\rm Ls}_{2}^{(1)}(\sigma)$ and ${\rm Ls}_{1,1}^{(0,0)}(\sigma)$ can be evaluated by the definition trivially. In the case $\mathbf{k}=(k)$ and $\mathbf{l}=(k-2)$, the log-sine integral $\Ls_{k}^{(k-2)}(\sigma)$ has the following easy expression:
\begin{align*}
\Ls_{k}^{(k-2)}(\sigma)&=\frac{i}{2k}\sigma^{k}-\frac{i\pi}{2(k-1)}\sigma^{k-1}+i^{k-1}(k-2)!\zeta(k)\\
&\quad-\sum_{j=0}^{k-2}i^{j+1}\frac{(k-2)!}{(k-2-j)!}\sigma^{k-2-j}\Li_{j+2}(e^{i\sigma})
\end{align*}
which is equivalent to \cite[(7.51)]{L}.
\section{Evaluation of iterated log-sine integrals at general argument}\label{se:evalLsgeneral}
In this section, we discuss iterated log-sine integrals at general argument and  give a proof of Theorem \ref{th:Lsgeneral}.
\subsection{Proof of Theorem \ref{th:Lsgeneral}}
First, we note the following reflection formula:
\begin{lemma}\label{le:Lsrefrect}
We have
\begin{align*}
\Ls_{\mathbf{k}}^{\mathbf{l}}(\sigma)=(-1)^{|\mathbf{l}|+n}\Ls_{\mathbf{k}}^{\mathbf{l}}(-\sigma).
\end{align*}
\end{lemma}
This lemma is given by changing of variables $\theta_{u}=-\theta_{u}$. 
By Lemma \ref{le:Lsrefrect}, we only need to prove Theorem \ref{th:Lsgeneral} for $2m\pi+\sigma \ge 0$.

We introduce the following integrals for convenience.
\begin{definition}
For $\rho, \sigma\in\mathbb{R}$, $\mathbf{k} = (k_{1}, \dots, k_{n}) \in \mathbb{N}^{n}$ and $\mathbf{l} = (l_{1},\dots,l_{n}) \in (\mathbb{Z}_{\ge 0})^{n}$, we define
\[\Ls_{\mathbf{k}}^{\mathbf{l}}(\rho; \sigma)=(-1)^{n}\int_{\rho}^{\sigma}\int_{\rho}^{\theta_{n}}\dots\int_{\rho}^{\theta_{2}}\prod_{u=1}^{n}\theta_u^{l_u}A^{k_u-1-l_u}(\theta_{u})\,d\theta_{1}\cdots d\theta_{n}.\]
\end{definition}
Then, the following lemmas hold. 
\begin{lemma}\label{le:decoform}Let $0<\rho<\sigma$, $\mathbf{k} = (k_{1}, \dots, k_{n}) \in \mathbb{N}^{n}$, $\mathbf{l} = (l_{1},\dots,l_{n}) \in (\mathbb{Z}_{\ge 0})^{n}$, and
\begin{align*}
\mathbf{k}_{h} &= (k_{1},\dots,k_{h}),\ \quad\mathbf{k}^{h} = (k_{h+1},\dots,k_{n}),\\
\mathbf{l}_{h} &= (l_{1},\dots,l_{h}),\qquad\mathbf{l}^{h} = (l_{h+1},\dots,l_{n}).
\end{align*}
If $\Ls_{\mathbf{k}}^{\mathbf{l}}(\sigma)$ converges absolutely, then
\begin{align*}
\Ls_{\mathbf{k}}^{\mathbf{l}}(\sigma)=\sum_{h=0}^{n}\Ls_{\mathbf{k}_{h}}^{\mathbf{l}_{h}}(\rho)\Ls_{\mathbf{k}^{h}}^{\mathbf{l}^{h}}(\rho;\sigma).
\end{align*}
\end{lemma}
This decomposition formula is proved by decomposing the domain of the integral \begin{align*}
0<\theta_{1}<\theta_{2}<\cdots<\theta_{n}<\sigma
\end{align*}into
\begin{align*}
0&<\rho<\theta_{1}<\theta_{2}<\cdots<\theta_{n}<\sigma,\\
0&<\theta_{1}<\rho<\theta_{2}<\cdots<\theta_{n}<\sigma,\\
&\qquad\qquad\qquad\vdots\\
0&<\theta_{1}<\theta_{2}<\cdots<\theta_{n}<\rho<\sigma.
\end{align*}
We write $\mathbf{j}\le\mathbf{k}$ when $\mathbf{j} = (j_{1},\dots,j_{n})$ and $\mathbf{k} = (k_{1},\dots,k_{n})$ satisfy $j_{1}\le k_{1},j_{2}\le k_{2},\dots,j_{n}\le k_{n}$,
and define
\begin{align*}
\binom{\mathbf{l}}{\mathbf{j}}=\prod_{u=1}^{n}\binom{l_{u}}{j_{u}}.
\end{align*}
\begin{lemma}\label{le:Lsshift}
For $m \in \mathbb{Z}$, we have
\begin{align*}
\Ls_{\mathbf{k}}^{\mathbf{l}}(2m\pi, 2m\pi+\sigma)=\sum_{\mathbf{j}\le\mathbf{l}}(2m\pi)^{|\mathbf{j}|}\binom{\mathbf{l}}{\mathbf{j}}\Ls_{\mathbf{k}-\mathbf{j}}^{\mathbf{l}-\mathbf{j}}(\sigma),
\end{align*}
where the sum is over all $\mathbf{j}=(j_{1},\dots,j_{n}) \in (\mathbb{Z}_{\ge0})^{n}$ satisfying $\mathbf{j}\le\mathbf{l}$.
\end{lemma}
\begin{proof}
By changing of variables $\theta_{u}=2m\pi-\theta_{u}$, we have
\begin{align*}
&\Ls_{\mathbf{k}}^{\mathbf{l}}(2m\pi, 2m\pi+\sigma)\\
&=(-1)^{n}\int_{2m\pi}^{2m\pi+\sigma}\int_{2m\pi}^{\theta_{n}}\dots\int_{2m\pi}^{\theta_{2}}\prod_{u=1}^{n}\theta_u^{l_u}A^{k_u-1-l_u}(\theta_{u})\,d\theta_{1}\cdots d\theta_{n}\\
&=(-1)^{n}\int_{0}^{-\sigma}\int_{0}^{\theta_{n}}\dots\int_{0}^{\theta_{2}}
\prod_{u=1}^{n}-\left(2m\pi-\theta_u\right)^{l_u}A^{k_u-1-l_u}(\theta_{u})\,d\theta_{1}\cdots d\theta_{n}\\
&=\int_{0}^{-\sigma}\int_{0}^{\theta_{n}}\dots\int_{0}^{\theta_{2}}
\prod_{u=1}^{n}\sum_{j_{u}=0}^{l_{u}}(2m\pi)^{j_{u}}\binom{l_{u}}{j_{u}}\left(-\theta_u\right)^{l_u-j_{u}}A^{k_u-1-l_u}(\theta_{u})\,d\theta_{1}\cdots d\theta_{n}\\
&=\sum_{\mathbf{j}\le\mathbf{l}}(2m\pi)^{|\mathbf{j}|}\binom{\mathbf{l}}{\mathbf{j}}(-1)^{|\mathbf{l}-\mathbf{j}|+n}\Ls_{\mathbf{k}-\mathbf{j}}^{\mathbf{l}-\mathbf{j}}(-\sigma).
\end{align*}
Here, we use Lemma \ref{le:Lsrefrect} and obtain the desired formula.
\end{proof}
\begin{lemma}\label{le:ls2mpitozeta}
Iterated log-sine integrals $\Ls_{\mathbf{k}}^{\mathbf{l}}(2m\pi)$ satisfying $k_{u} -1 - l_{u} \ge 0$ for all $u \in \{1,\dots,n\}$ can be written as a $\mathbb{Q}$-linear combination of products of a power of $\pi$ and multiple zeta values.
\end{lemma}
We note that the assertion for the generalized log-sine integrals $\Ls_{k}^{(l)}(2m\pi)$ was proved by Borwein and Straub \cite[Section 4.1]{BS}.
\begin{proof}
We assume $m\ge1$ and prove this lemma by induction on $m$. In the case  $m=1$, we obtain the assertion by putting $\sigma=2\pi$ in (\ref{eq:main}), since
$\Li_{\mathbf{k}}(e^{2\pi i})=\zeta(\mathbf{k})$ when $k_{n}\ge2$. We assume the assertion for $2m\pi$ and prove for $2(m+1)\pi$. By putting $\rho=2m\pi$ and $\sigma=2(m+1)\pi$ in Lemma \ref{le:decoform}, we obtain
\begin{align*}
\Ls_{\mathbf{k}}^{\mathbf{l}}(2(m+1)\pi)&=\sum_{j=0}^{n}\Ls_{\mathbf{k}_{h}}^{\mathbf{l}_{h}}(2m\pi)\Ls_{\mathbf{k}^{h}}^{\mathbf{l}^{h}}(2m\pi;2(m+1)\pi)\\
&=\sum_{j=0}^{n}\Ls_{\mathbf{k}_{h}}^{\mathbf{l}_{h}}(2m\pi)\sum_{\mathbf{j}\le\mathbf{l}^{h}}\binom{\mathbf{l}^{h}}{\mathbf{j}}(2m\pi)^{|\mathbf{j}|}\Ls_{\mathbf{k}^{h}-\mathbf{j}}^{\mathbf{l}^{h}-\mathbf{j}}(2\pi).
\end{align*}
Here, the second equality is obtained by putting $\sigma=2\pi$ in Lemma \ref{le:Lsshift}. Therefore, the assertion is proved.
\end{proof}
\begin{proof}[Proof of Theorem \ref{th:Lsgeneral}]
Let $\sigma\in[0,2\pi]$ and $m\ge0$. By putting $\rho=2m\pi$ and $\sigma=2m\pi+\sigma$ in Lemma \ref{le:decoform}, we have
\begin{align*}
\Ls_{\mathbf{k}}^{\mathbf{l}}(2m\pi+\sigma)=\sum_{j=0}^{n}\Ls_{\mathbf{k}_{h}}^{\mathbf{l}_{h}}(2m\pi)\Ls_{\mathbf{k}^{h}}^{\mathbf{l}^{h}}(2m\pi;2m\pi+\sigma).
\end{align*}
Here, $\Ls_{\mathbf{k}_{h}}^{\mathbf{l}_{h}}(2m\pi)$ can be written in terms of multiple zeta values by Lemma \ref{le:ls2mpitozeta}, and $\Ls_{\mathbf{k}^{h}}^{\mathbf{l}^{h}}(2m\pi,2m\pi+\sigma)$ can be written in terms of multiple zeta values and multiple polylogarithms with admissible index at $e^{i\sigma}$ by Lemma \ref{le:Lsshift} and Theorem \ref{th:evalLs}. Therefore, Theorem \ref{th:Lsgeneral} is proved.
\end{proof}
\subsection{Examples}
For $\sigma \in [0, 2\pi]$, we have the following evaluations:
\begin{align*}
{\rm Ls}_{1}^{(0)}(2\pi+\sigma)&={\rm Ls}_{1}^{(0)}\left(2\pi\right) + {\rm Ls}_{1}^{(0)}\left(\sigma\right)\\
&=-2\pi-\sigma,\\
{\rm Ls}_{2}^{(0)}(2\pi+\sigma)&={\rm Ls}_{2}^{(0)}\left(2\pi\right) + {\rm Ls}_{2}^{(0)}\left(\sigma\right)\\
&=-\frac{1}{2} i  \pi \sigma + \frac{1}{4} i  \sigma^{2} + i  \zeta\left(2\right) - i  {\rm Li}_{2}\left(e^{i  \sigma}\right),\\
{\rm Ls}_{2}^{(1)}(2\pi+\sigma)
&={\rm Ls}_{2}^{(1)}\left(2\pi\right) + 2\pi {\rm Ls}_{1}^{(0)}\left(\sigma\right) + {\rm Ls}_{2}^{(1)}\left(\sigma\right)\\
&=-2  \pi^{2} - 2  \pi \sigma - \frac{1}{2}  \sigma^{2},\\
{\rm Ls}_{1 , 1}^{(0 , 0)}(2\pi+\sigma)&={\rm Ls}_{1 , 1}^{(0 , 0)}\left(2\pi\right) + {\rm Ls}_{1}^{(0)}\left(2\pi\right) {\rm Ls}_{1}^{(0)}\left(\sigma\right) + {\rm Ls}_{1 , 1}^{(0 , 0)}\left(\sigma\right)\\
&=2  \pi^{2} + 2  \pi \sigma + \frac{1}{2}  \sigma^{2}.
\end{align*}
\section{Conjectures}\label{se:Con}
We introduce and discuss shifted log-sine integrals. Then, we suggest some conjectures on shifted log-sine integrals, multiple zeta values, multiple Clausen values, multiple Glaisher values and iterated log-sine integrals.
\subsection{Shifted log-sine integrals}\label{se:SLs}
Shifted log-sine integrals are defined as a kind of log-sine integrals as follows.
\begin{definition}[Shifted log-sine integrals]
For $\sigma \in \mathbb{R}$, $\mathbf{k} = (k_{1}, \dots, k_{n}) \in (\mathbb{Z}_{\ge 2})^{n}$, we define
\[\SLs(\mathbf{k};\sigma)=\int_{0}^{\sigma}\int_{0}^{\theta_{n}}\dots\int_{0}^{\theta_{2}}\prod_{u=1}^{n}
\left(\theta_{u}-\sigma\right)^{k_{u}-2}A(\theta_{u})\,d\theta_{1}\cdots d\theta_{n}.\]
\end{definition}

Shifted log-sine integrals satisfy the following shuffle product.
\begin{proposition}\label{pr:SLsshuffle}
For $n, m \in \mathbb{N}$ and $k_{1},\dots,k_{n+m} \in \mathbb{Z}_{\ge 2}$, we have
\begin{align*}&\SLs(k_{1},\dots,k_{n};\sigma)\cdot\SLs(k_{n+1},\dots,k_{n+m};\sigma)\\
&=\sum_{\tau \in \mathfrak{S}_{n,m}}\SLs(k_{\tau(1)},\dots,k_{\tau(n+m)};\sigma),
\end{align*}
where $\mathfrak{S}_{n,m}$ be a subset of $\mathfrak{S}_{n+m}$ $($the symmetric group of degree $n+m$$)$,  which is defined by 
\[\mathfrak{S}_{n,m} = \{\tau \in \mathfrak{S}_{n+m} \mid \tau^{-1}(1)<\dots<\tau^{-1}(n)\ and \ \tau^{-1}(n+1)<\dots<\tau^{-1}(n+m)\}.\]
\end{proposition}
This proposition is proved by decomposing the domain of the integral.

Shifted log-sine integrals can be written in terms of iterated log-sine integrals as follows.
\begin{theorem}\label{th:evalSLs}
Let 
\[\mathbf{2}_{n}=(\{2\}^{n}).\]
For $\sigma \in \mathbb{R}$ and $\mathbf{k} = (k_{1}, \dots, k_{n}) \in (\mathbb{Z}_{\ge 2})^{n}$, we have
\[\SLs(\mathbf{k};\sigma)=(-1)^{n}\sum_{\mathbf{j}\le\mathbf{k}-\mathbf{2}_{n}}\binom{\mathbf{k}-\mathbf{2}_{n}}{\mathbf{j}}(-\sigma)^{|\mathbf{k}-\mathbf{2}_{n}-\mathbf{j}|}\Ls_{\mathbf{j}+\mathbf{2}_{n}}^{\mathbf{j}}(\sigma),\]
where the sum is over all $\mathbf{j} = (j_{1},\dots,j_{n}) \in (\mathbb{Z}_{\ge 0})^{n}$ satisfying $\mathbf{j}\le\mathbf{k}-\mathbf{2}_{n}$.
\end{theorem}
\begin{proof}
We calculate the integrand $\prod_{u=1}^{n}\left(\theta_{u}-\sigma\right)^{k_{u}-2}A(\theta_{u})$ of the definition of the shifted log-sine integral.
Then, we have
\begin{align*}
\prod_{u=1}^{n}
\left(\theta_{u}-\sigma\right)^{k_{u}-2}A(\theta_{u})&=\prod_{u=1}^{n}
\sum_{j_{u}=0}^{k_{u}-2}\binom{k_{u}-2}{j_{u}}\theta_{u}^{j_{u}}\left(-\sigma\right)^{k_{u}-2-j_{u}}A(\theta_{u})\\
&=\sum_{\mathbf{j}\le\mathbf{k}-\mathbf{2}_{n}}\binom{\mathbf{k}-\mathbf{2}_{n}}{\mathbf{j}}(-\sigma)^{|\mathbf{k}-\mathbf{2}_{n}-\mathbf{j}|}
\prod_{u=1}^{n}\theta_{u}^{j_{u}}A(\theta_{u}).
\end{align*}
Therefore, we obtain the desired formula.
\end{proof}
\begin{conjecture}
Elements of the set 
\begin{align*}\left\{\pi^{m}\SLs(k_{1},\dots,k_{n}) \mathrel{}\middle|\mathrel{} m \ge 0,  n \ge 0, k_{i} \ge 2\right\}
\end{align*}
are $\mathbb{Q}$-linearly independent.
\end{conjecture}
The author calculated numerical values of shifted log-sine integrals by using Theorem \ref{th:evalSLs} and Theorem \ref{th:evalLs}. Then, by using the lindep of Pari/GP, there seems to be no $\mathbb{Q}$-linear relations among these values satisfying $m+k_{1}+\dots+k_{n}\le8$.
\subsection{Conjectures on multiple zeta values}\label{se:conjMZVs}
We define $\mathcal{Z}_{k}$ as the vector space spanned by all multiple zeta values of weight $k$. Namely,
\begin{align*}
\mathcal{Z}_{k} &= \sum_{\substack{0 \le n \le k \\ k_{1}+\dots+k_{n}=k \\ k_{1},\dots,k_{n-1} \ge 1, k_{n} \ge 2}}\mathbb{Q}\cdot\zeta(k_{1},\dots,k_{n}).
\end{align*}
We write $\SLs(\mathbf{k}) := \SLs(\mathbf{k};\pi/3)$ and define 
\begin{align*}S'_{k,d}=\left\{\pi^{2m}\SLs(k_{1},\dots,k_{n}) \mathrel{}\middle|\mathrel{} \begin{gathered}
2m+k_{1}+\dots+k_{n}=k,\\m \ge 0,  d \ge n \ge 0, \\k_{i} \ge 3:{\rm odd}.\end{gathered}\right\}
\end{align*}
for $k\in\mathbb{Z}_{\ge0}$ and $d\in\mathbb{Z}_{\ge0}$.
\begin{conjecture}\label{con:genZ}
Every multiple zeta value with weight $k$ and depth $d$ can be written as a $\mathbb{Q}$-linear combination of elements of $S'_{k,d}$.
\end{conjecture}
\begin{conjecture}\label{con:basisZ}
The set of the real numbers $S'_{k,k}$ is a basis of $\mathcal{Z}_{k}$.
\end{conjecture}
Zagier \cite{Z} conjectured $\dim{\mathcal{Z}_{k}}=d_{k}$, where $d_{k}$ is a sequence defined by 
\begin{align*}
d_{0}&=1,\,d_{1} = 0,\,d_{2} = 1,\\
d_{k}&= d_{k-2} + d_{k-3}\quad(k \ge 3).
\end{align*}
If Conjecture \ref{con:basisZ} is true, then Zagier's conjecture is true, because the numbers $|S'_{k,d}|$ of elements of $S'_{k,d}$ satisfy
\begin{align*}
|S^{\prime}_{0,d}| &= 1,\,|S^{\prime}_{1,d}| = 0,\,|S^{\prime}_{2,d}| = 1\quad(d\ge0),\\
|S^{\prime}_{k,d}| &= |S^{\prime}_{k-2,d}| + |S^{\prime}_{k-3,d-1}|\quad(k\ge3, d\ge1).
\end{align*}
We consider the set of multiple zeta values
\begin{align*}H_{k,d}=\left\{\zeta(k_{1},\dots,k_{n}) \mathrel{}\middle|\mathrel{} \begin{gathered}
k_{1}+\dots+k_{n}=k,\ n \le k,\\ \ k_{i} \in \{2,3\},\ |\{ i \mid k_{i} = 3\}| \le d\end{gathered}\right\}.
\end{align*}
Hoffman \cite{H} conjectured that the set of multiple zeta values $H_{k,k}$ is a basis of $\mathcal{Z}_{k}$. Conjecture \ref{con:basisZ} can be regarded as an analogue of Hoffman's conjecture.
It follows by Brown \cite{B} and Goncharov \cite[Theorem 1.2]{G} that every multiple zeta value with weight $k$ and depth $d$ can be written as a $\mathbb{Q}$-linear combination of elements of $H_{k,d}$. Conjecture \ref{con:genZ} can be regarded as an analogue of this assertion. However, the real numbers in $S'_{k,d}$ have different properties from those of the multiple zeta values in $H_{k,d}$. We define 
\begin{align*}
\mathcal{S}'_{k,d}=\text{span}_{\mathbb{Q}}(S'_{k,d})\ \ \text{and}\ \ 
\mathcal{H}_{k,d}=\text{span}_{\mathbb{Q}}(H_{k,d}).
\end{align*}
Then, we can see that $\mathcal{S}'_{k_{1},d_{1}} \cdot \mathcal{S}'_{k_{2}, d_{2}} \subset \mathcal{S}'_{k_{1}+k_{2},d_{1}+d_{2}}$ just by using the shuffle product formula (Proposition \ref{pr:SLsshuffle}). More strictly, we can see that the product of an element of $S'_{k_{1},d_{1}}$ and an element of $S'_{k_{2},d_{2}}$ can be written as the sum of products of a positive integer and an element of $S'_{k_{1}+k_{2}, d_{1}+d_{2}}$. However, we cannot see that $\mathcal{H}_{k_{1},d_{1}} \cdot \mathcal{H}_{k_{2}, d_{2}} \subset \mathcal{H}_{k_{1}+k_{2},d_{1}+d_{2}}$ just by using the harmonic product formula or the shuffle product formula. 
For example, each products of two $\zeta(2) \in H_{2,0}$ are
\begin{align*}
\zeta(2)\cdot\zeta(2) =\begin{cases} 2\zeta(2,2)+\zeta(4)\quad&{\rm (harmonic \ product)},\\
2\zeta(2,2)+4\zeta(1,3)\quad&{\rm (shuffle\ product)}.\end{cases}
\end{align*}
Here, $\zeta(2,2) \in H_{4,0}$, but $\zeta(4) \notin H_{4,0}$ and $\zeta(1,3) \notin H_{4,0}$.

The author checked Conjecture \ref{con:genZ} and Conjecture \ref{con:basisZ} up to weight $13$ by numerical experiments.

Theoretically, we can check that some multiple zeta values can be written as a $\mathbb{Q}$-linear combination of elements of $S'_{k,d}$. First, we consider the Riemann zeta values.
\begin{theorem}\label{th:rzetainS}
The Riemann zeta value $\zeta(k)$ can be written as a $\mathbb{Q}$-linear combination of elements of $S'_{k,1}$.
\end{theorem}
\begin{proof}
When the argument is even, it is known that $\zeta(2k)\in\mathbb{Q}\cdot\pi^{2k}$. Therefore, we can see that $\zeta(2k)$ can be written as a $\mathbb{Q}$-linear combination of elements of $S'_{2k,1}=S'_{2k,0}=\{\pi^{2k}\}$.

When the argument is odd, the following theorem is known.
\begin{theorem}[{Choi--Cho--Srivistava \cite[(4.14)]{CCS}(Lewin  \cite[(7.160)]{L})}]
For $k \in \mathbb{Z}_{\ge 0}$, we have
\begin{align}
&(-1)^{k}\int_{0}^{\frac{\pi}{3}}\left(\theta-\frac{\pi}{3}\right)^{2k+1}A(\theta)\,d\theta\label{eq:CCS}\\
&=-\frac{1}{2}(2k+1)!(1-2^{-2k-2})(1-3^{-2k-2})\zeta(2k+3)\nonumber\\
&\quad+(2k+1)!\sum_{m=0}^{k}(-1)^{m}\left(\frac{\pi}{3}\right)^{2m}\frac{\zeta(2k+3-2m)}{(2m)!}.\nonumber
\end{align}
\end{theorem}
We show that $\zeta(2k+3)$ can be written as a $\mathbb{Q}$-linear combination of elements of $S'_{2k+3,1}$ by induction on $k$.
By putting $k=0$ in (\ref{eq:CCS}), we obtain $\zeta(3)=(3/2)\SLs(3)$. Therefore, we can see that $\zeta(3)$ can be written as a $\mathbb{Q}$-linear combination of elements of $S'_{3,1}=\{\SLs(3)\}$. By (\ref{eq:CCS}), we have
\begin{align*}
&-\frac{1}{2}(2k+1)!\left((1-2^{-2k-2})(1-3^{-2k-2})-2\right)\zeta(2k+3)\\
&=(-1)^{k}\SLs(2k+3)\\
&\quad-(2k+1)!\sum_{m=1}^{k}(-1)^{m}\left(\frac{\pi}{3}\right)^{2m}\frac{\zeta(2k+3-2m)}{(2m)!}.
\end{align*}
We assume $\zeta(2k+3-2m)$ with $m\in\{1,\dots,k\}$ can be written as a $\mathbb{Q}$-linear combination of elements of $S'_{2k+3-2m,1}$. Then, we can see that $\zeta(2k+3)$ can be written as a $\mathbb{Q}$-linear combination of elements of $S'_{2k+3,1}$ because the product of $\pi^{2m}$ and an element of $S'_{2k+3-2m,1}$ is included in $S'_{2k+3,1}$.
\end{proof}
By Theorem \ref{th:rzetainS} and the shuffle product formula, multiple zeta values such that
\[\zeta(\mathbf{k})\in\sum_{\substack{2m+l_{1}+\dots+l_{h}=|\mathbf{k}|\\m\ge0, \ 0\le h \le d}}\mathbb{Q}\cdot\pi^{2m}\cdot\zeta(l_{1})\cdots\zeta(l_{h})\]
can be written as a $\mathbb{Q}$-linear combination of elements of $S'_{|\mathbf{k}|,d}$.
In particular, multiple zeta values of the weight up to $7$ can be written as a $\mathbb{Q}$-linear combination of elements of $S'_{k,k}$.
\subsection{Conjectures on multiple Clausen values and multiple Glaisher values}\label{se:conjMCVs}
We define $\mathcal{C}_{k}$ as the vector space spanned by all multiple Clausen values of weight $k$, and $\mathcal{G}_{k}$ as the vector space spanned by all multiple Glaisher values of weight $k$. Namely,
\begin{align*}
\mathcal{C}_{k} &= \sum_{\substack{0 \le n \le k \\ k_{1}+\dots+k_{n}=k \\ k_{1},\dots,k_{n} \ge 1}}\mathbb{Q}\cdot {\rm Cl}_{k_{1},\dots,k_{n}}\left(\pi/3\right),\\
\mathcal{G}_{k} &= \sum_{\substack{0 \le n \le k \\ k_{1}+\dots+k_{n}=k \\ k_{1},\dots,k_{n} \ge 1}}\mathbb{Q}\cdot {\rm Gl}_{k_{1},\dots,k_{n}}\left(\pi/3\right).\\
\end{align*}
For $k\in\mathbb{Z}_{\ge0}$ and $d\in\mathbb{Z}_{\ge0}$, we define
\begin{align*}
S^{o}_{k,d}&=\left\{\pi^{m}\SLs(k_{1},\dots,k_{n}) \mathrel{}\middle|\mathrel{} \begin{gathered}
m+k_{1}+\dots+k_{n}=k,\\d\ge n \ge 0,\ n:{\rm odd},\\m \ge 0, k_{i} \ge 2.\end{gathered}\right\},\\
S^{e}_{k,d}&=\left\{\pi^{m}\SLs(k_{1},\dots,k_{n}) \mathrel{}\middle|\mathrel{} \begin{gathered}
m+k_{1}+\dots+k_{n}=k,\\d \ge n \ge 0,\ n:{\rm even},\\m \ge 0, k_{i} \ge 2.\end{gathered}\right\}.
\end{align*}
\begin{conjecture}\label{con:genCG}
The following assertions hold.
\begin{enumerate}[(i)]
\item Every multiple Clausen values with weight $k$ and depth $d$ can be written as a $\mathbb{Q}$-linear combination of elements of $S^{o}_{k,d}$.
\item Every multiple Glaisher values with weight $k$ and depth $d$ can be written as a $\mathbb{Q}$-linear combination of elements of $S^{e}_{k,d}$.
\end{enumerate}
\end{conjecture}
\begin{conjecture}\label{con:basisCG}
The following assertions hold.
\begin{enumerate}[(i)]
 \item The set of the real numbers $S^{o}_{k,k}$ is a basis of $\mathcal{C}_{k}$.
 \item The set of the real numbers $S^{e}_{k,k}$ is a basis of $\mathcal{G}_{k}$.
 \item The set of the real numbers $S^{o}_{k,k}\cup S^{e}_{k,k}$ is a basis of $\mathcal{C}_{k}+\mathcal{G}_{k}$.
\end{enumerate}
\end{conjecture}
Borwein, Broadhurst and Kamnitzer \cite{BBK} conjectured  $\dim{\mathcal{C}_{k}} = I(k)$ and $\dim{\mathcal{G}_{k}} = R(k)$, where $R(k)$ and $I(k)$ are sequences defined by
\begin{align*}
I(0)&=I(1)=0,\quad R(0)=R(1)=1,\\
I(k) &= I(k-1) + R(k-2)\quad(k\ge2),\\
R(k) &= R(k-1) + I(k-2)\quad(k\ge 2).
\end{align*}
Note that $W(k):=I(k)+R(k)$ is equal to $F_{k+1}$, the $(k+1)$-th Fibonacci number. 
If Conjecture \ref{con:basisCG} is true, then the Borwein-Broadhurst-Kamnitzer conjecture is true, because $|S^{o}_{k,d}|$ and $|S^{e}_{k,d}|$ satisfy
\begin{align*}
|S^{o}_{0,0}| &= |S^{o}_{1,1}|=0, \quad |S^{e}_{0,0}| = |S^{e}_{1,1}|=1,\\
|S^{o}_{k,d}| &= |S^{o}_{k-1,d}| + |S^{e}_{k-2,d-1}|\quad(k\ge2, d\ge1),\\
|S^{e}_{k,d}| &= |S^{e}_{k-1,d}| + |S^{o}_{k-2,d-1}|\quad(k\ge2, d\ge1).
\end{align*}

The author checked Conjecture \ref{con:genCG} and Conjecture \ref{con:basisCG} up to weight $9$ by numerical experiments.

Theoretically, we can check that some multiple Clausen values and multiple Glaisher values can be written as a $\mathbb{Q}$-linear combination of elements of $S^{o}_{k,d}$ and $S^{e}_{k,d}$, respectively. First, we consider the case when the depth is $1$.
\begin{theorem}\label{th:CGinS}The following assertions hold.
\begin{enumerate}[(i)]
 \item The Clausen value $\Cl_{k}(\pi/3)$ can be written as a $\mathbb{Q}$-linear combination of elements of $S^{o}_{k,1}$.
 \item The Glaisher value $\Gl_{k}(\pi/3)$ can be written as a $\mathbb{Q}$-linear combination of elements of $S^{e}_{k,1}(=S^{e}_{k,0})$.
\end{enumerate}
\end{theorem}
\begin{proof}
Clausen values of odd index are evaluated as follows 
(see \cite[p.198]{L}):
\begin{align*}\Cl_{2k+1}(\pi/3) &= \frac{1}{2}(1-2^{-2k})(1-3^{-2k})\zeta(2k+1).
\end{align*}
By Theorem \ref{th:rzetainS} and $S'_{2k+1,1} \subset {S}^{o}_{2k+1, 1}$, the assertion $(i)$ for odd indices holds.

On even index, we use the following formula for $\SLs(2k+2)$:
\begin{align*}
\SLs(2k+2)&=\frac{i}{2k+1}\left(-\frac{\pi}{3}\right)^{2k+2}+\frac{i}{2(2k+2)}\left(-\frac{\pi}{3}\right)^{2k+2}\\
&+i(-1)^{k}(2k)!\Li_{2+2k}(e^{i\pi/3})-i(-1)^{k}(2k)!\sum_{j=0}^{2k}\frac{\left(i\pi/3\right)^{2k-j}}{(2k-j)!}\zeta(2+j),
\end{align*}
which is obtained by Theorem \ref{th:evalSLs} and Theorem \ref{th:evalLs}. Therefore, we obtain
\begin{align*}
&\Cl_{2+2k}(\pi/3)\\
&=-\frac{(-1)^{k}}{(2k)!}\SLs(2k+2)+\sum_{j=0}^{k-1}(-1)^{k-j-1}\frac{\left(\pi/3\right)^{2k-2j-1}}{(2k-2j-1)!}\zeta(3+2j)
\end{align*}
and obtain the assertion $(i)$.

On the other hand, the following explicit evaluation of Glaisher functions is known (see \cite[(7.60)]{L}, for example):
\begin{align*}
\Gl_{k}(2\pi x) = (-1)^{1+[k/2]}2^{k-1}\pi^{k}B_{k}(x)/k!\qquad(0\le x\le1, k>1),
\end{align*}
where $B_{n}(x)$ denotes the $n$-th Bernoulli polynomial.
Therefore, we have
\begin{align*}
\Gl_{k}\left(\pi/3\right) = (-1)^{1+[k/2]}2^{k-1}\pi^{k}B_{k}(1/6)/k!.
\end{align*}
Since ${S}^{e}_{k,1}={S}^{e}_{k,0}=\{\pi^{k}\}$, we can see that Glaisher value $\Gl_{k}(\pi/3)$ can be written as a $\mathbb{Q}$-linear combination of elements of  $S^{e}_{k,1}$.
\end{proof}

We consider indices the form $\mathbf{k}=(\{1\}^{a},2,\{1\}^{b})$. By \cite[Theorem 2]{U}, we have
\begin{align*}
&\mathrm{Li}_{\{1\}^{a},2,\{1\}^{b}}(e^{\frac{\pi}{3}i})\\
&=i^{a+b+1}\int_{0<\theta_{1}<\dots<\theta_{a+b+1}<\frac{\pi}{3}} A(\theta_{a+2}) - A(\theta_{a+1}) - \frac{i\theta_{a+2}}{2} + \frac{i\theta_{a+1}}{2}\,d\theta_{1}\cdots d\theta_{a+b+1}\\
&=i^{a+b+1}\frac{1}{(a+1)!}\frac{1}{(b-1)!}\int_{0}^{\frac{\pi}{3}}\theta^{a+1}A(\theta)\left(\frac{\pi}{3}-\theta\right)^{b-1}\,d\theta\\
&\quad-i^{a+b+1}\frac{1}{a!}\frac{1}{b!}\int_{0}^{\frac{\pi}{3}}\theta^{a}A(\theta)\left(\frac{\pi}{3}-\theta\right)^{b}\,d\theta\\
&\quad+\frac{i^{a+b}}{2}\left(\frac{a+2}{(a+b+2)!}\left(\frac{\pi}{3}\right)^{a+b+2}-\frac{a+1}{(a+b+2)!}\left(\frac{\pi}{3}\right)^{a+b+2}\right)\\
&=i^{a+b+1}\sum_{j=0}^{a+1}\frac{1}{(a+1)!}\frac{(-1)^{b-1}}{(b-1)!}\binom{a+1}{j}\left(\frac{\pi}{3}\right)^{j}\SLs(a+b+2-j)\\
&\quad-i^{a+b+1}\sum_{j=0}^{a}\frac{1}{a!}\frac{(-1)^{b}}{b!}\binom{a}{j}\left(\frac{\pi}{3}\right)^{j}\SLs(a+b+2-j)+i^{a+b}\frac{(\pi/3)^{a+b+2}}{2(a+b+2)!}
\end{align*}
for $a\ge0$ and $b\ge1$, and similarly
\begin{align*}
&\mathrm{Li}_{\{1\}^{a},2}(e^{\frac{\pi}{3}i})=-i^{a+1}\sum_{j=0}^{a}\frac{1}{a!}\binom{a}{j}\left(\frac{\pi}{3}\right)^{j}\SLs(a+2-j)+i^{a}\frac{(\pi/3)^{a+2}}{2(a+2)!}
\end{align*}
for $a\ge0$.
Therefore,  for $a\ge0$ and $b\ge0$, we can see that $\Cl_{\{1\}^{a},2,\{1\}^{b}}(\pi/3)$ and $\Gl_{\{1\}^{a},2,\{1\}^{b}}(\pi/3)$ can be written as a $\mathbb{Q}$-linear combination of elements of $S^{o}_{a+b+2,1}$ and $S^{e}_{a+b+2,1}(=S^{e}_{a+b+2,0})$, respectively.
\section{On iterated log-sine integrals}\label{se:ILSs}
We define $\mathcal{L}^{o}_{k}$ as the vector space spanned by all iterated log-sine integrals at $\pi/3$ whose weight is $k$ and $\left|\mathbf{k}-\mathbf{1}_{n}-\mathbf{l}\right|$ is an odd number, where we note that $\left|\mathbf{k}-\mathbf{1}_{n}-\mathbf{l}\right|$ is the sum of exponents of $A(\theta):=\log\left|2\sin(\theta/2)\right|$. Similarly, we define $\mathcal{L}^{e}_{k}$ as the vector space spanned by all iterated log-sine integrals at $\pi/3$ whose weight is $k$ and $\left|\mathbf{k}-\mathbf{1}_{n}-\mathbf{l}\right|$ is an even number. Namely,
\begin{align*}
\mathcal{L}^{o}_{k} &= \sum_{\substack{k_{1}+\dots+k_{n}=k, 0 \le n \le k \\ k_{1},\dots,k_{n} \ge 1, l_{1},\dots,l_{n} \ge 0\\ k_{1} -1-l_{1},\dots,k_{n} -1-l_{n} \ge 0 \\\sum_{u=1}^{n}k_{u}-1-l_{u}:\text{odd}}}\mathbb{Q}\cdot \Ls_{k_{1},\dots,k_{n}}^{(l_{1}, \dots, l_{n})}(\pi/3),\\
\mathcal{L}^{e}_{k} &= \sum_{\substack{k_{1}+\dots+k_{n}=k, 0 \le n \le k \\ k_{1},\dots,k_{n} \ge 1, l_{1},\dots,l_{n} \ge 0\\ k_{1} -1-l_{1},\dots,k_{n} -1-l_{n} \ge 0\\\sum_{u=1}^{n}k_{u}-1-l_{u}:\text{even}}}\mathbb{Q}\cdot \Ls_{k_{1},\dots,k_{n}}^{(l_{1}, \dots, l_{n})}(\pi/3),\\
\end{align*}
where we understand $\Ls_{\emptyset}^{\emptyset}(\sigma)$ satisfies $\sum_{u=1}^{n}k_{u}-1-l_{u}=0$.

Before stating conjectures on iterated log-sine integrals, we count the number of generators of $\mathcal{L}^{o}_{k}$ and $\mathcal{L}^{e}_{k}$. For non-negative integers $k$, $d$ and a fixed real number $\sigma$, we define
\begin{align*}
L^{o}_{k,d}(\sigma) = \left\{\Ls_{k_{1},\dots,k_{n}}^{(l_{1},\dots,l_{n})}(\sigma) \mathrel{}\middle|\mathrel{} \begin{gathered}
k_{1}+\dots+k_{n}=k, 0 \le n \le k \\ k_{1},\dots,k_{n} \ge 1, l_{1},\dots,l_{n} \ge 0\\ k_{1} -1-l_{1},\dots,k_{n} -1-l_{n} \ge 0 \\\sum_{u=1}^{n}k_{u}-1-l_{u} \le d:odd.\end{gathered}\right\}
\end{align*}
and define $L^{e}_{k,d}(\sigma)$ analogously for the even case.
Then, the number of elements of those sets is evaluated as follows.
\begin{theorem}\label{th:Lnumber}We have
\begin{align}
&\left|L^{o}_{0,d}(\sigma)\right|=\left|L^{o}_{1,d}(\sigma)\right|=0,\ \ &&\left|L^{e}_{0,d}(\sigma)\right|=\left|L^{e}_{1,d}(\sigma)\right|=1,\label{eq:Minitialk}\\
&\left|L^{o}_{2,d}(\sigma)\right|=\begin{cases}0&(d=0),\\1&(d\ge1),\end{cases}\ \ &&\left|L^{e}_{2,d}(\sigma)\right|=2,\label{eq:Minitial2}\\
&\left|L^{o}_{k,0}(\sigma)\right|=0,\ &&\left|L^{e}_{k,0}(\sigma)\right|=
\begin{cases}1&(k=0),\\2^{k-1}&(k\ge1),\end{cases}\label{eq:Minitiald}
\end{align}
and
\begin{align}\label{eq:vrec}
\begin{split}
\left|L^{o}_{k,d}(\sigma)\right| &= 2\left|L^{o}_{k-1,d}(\sigma)\right|+\left|L^{e}_{k-1,d-1}(\sigma)\right|-\left|L^{e}_{k-2,d-1}(\sigma)\right|\quad(k\ge3, d\ge1),\\
\left|L^{e}_{k,d}(\sigma)\right| &= 2\left|L^{e}_{k-1,d}(\sigma)\right|+\left|L^{o}_{k-1,d-1}(\sigma)\right|-\left|L^{o}_{k-2,d-1}(\sigma)\right|\quad(k\ge3, d\ge1).
\end{split}
\end{align}
In particular, we have
\begin{align}\label{eq:vkk}
\begin{split}
\left|L^{o}_{0,0}(\sigma)\right| &=0,\quad\left|L^{e}_{0,0}(\sigma)\right|=1,\\
\left|L^{o}_{k,k}(\sigma)\right| &= (F_{2k}-F_{k})/2\quad(k\ge1),\\
\left|L^{e}_{k,k}(\sigma)\right| &= (F_{2k}+F_{k})/2\quad(k\ge1).
\end{split}
\end{align}
\end{theorem}
\begin{proof}
We have (\ref{eq:Minitialk}) since iterated log-sine integrals of weight $0$ is only $\Ls_{\emptyset}^{\emptyset}(\sigma)$ and iterated log-sine integrals of weight $1$ is only $\Ls_{1}^{(0)}(\sigma)$. We also have (\ref{eq:Minitial2}) since iterated log-sine integrals of weight $2$ are $\Ls_{2}^{(0)}(\sigma)$, $\Ls_{2}^{(1)}(\sigma)$ and $\Ls_{1,1}^{(0,0)}(\sigma)$. The number of elements of $L^{e}_{k,0}(\sigma)$ is equal to the number of indices of weight $k$ because iterated log-sine integrals satisfying $|\mathbf{k}-\mathbf{1}_{n}-\mathbf{l}|=0$ are of the form $\Ls_{k_{1},\dots,k_{n}}^{(k_{1}-1,\dots,k_{n}-1)}(\sigma)$. 
Therefore,  we obtain (\ref{eq:Minitiald}).

For $k\ge3$ and $d\ge1$, we have
\begin{align*}
\left|L^{o}_{k,d}(\sigma)\right|&=\left|\left\{\Ls_{k_{1},\dots,k_{n},1}^{(l_{1}, \dots, l_{n},0)}(\sigma) \mathrel{}\middle|\mathrel{} \Ls_{k_{1},\dots,k_{n}}^{(l_{1}, \dots, l_{n})}(\sigma) \in L^{o}_{k-1,d}(\sigma)\right\}\right|\\
&\quad+\left|\left\{\Ls_{k_{1},\dots,k_{n}+1}^{(l_{1}, \dots, l_{n}+1)}(\sigma) \mathrel{}\middle|\mathrel{} \Ls_{k_{1},\dots,k_{n}}^{(l_{1}, \dots, l_{n})}(\sigma) \in L^{o}_{k-1,d}(\sigma)\right\}\right|\\
&\quad+\left|\left\{\Ls_{k_{1},\dots,k_{n}+1}^{(l_{1}, \dots, l_{n})}(\sigma) \mathrel{}\middle|\mathrel{} \Ls_{k_{1},\dots,k_{n}}^{(l_{1}, \dots, l_{n})}(\sigma) \in  L^{e}_{k-1,d-1}(\sigma)\right\}\right|\\
&\quad-\left|\left\{\Ls_{k_{1},\dots,k_{n}+2}^{(l_{1}, \dots, l_{n}+1)}(\sigma) \mathrel{}\middle|\mathrel{} \Ls_{k_{1},\dots,k_{n}}^{(l_{1}, \dots, l_{n})}(\sigma) \in L^{e}_{k-2,d-1}(\sigma)\right\}\right|,
\end{align*}
and the corresponding formula holds for $|L^{e}_{k,d}(\sigma)|$. Therefore, we obtain (\ref{eq:vrec}).

By (\ref{eq:vrec}), for $k\ge3$, we obtain
\begin{align*}
&\left|L^{e}_{k,k}(\sigma)\right|+\left|L^{o}_{k,k}(\sigma)\right| \\
&= 3\left(\left|L^{e}_{k-1,k-1}(\sigma)\right|+\left|L^{o}_{k-1,k-1}(\sigma)\right|\right)-\left(\left|L^{e}_{k-2,k-2}(\sigma)\right|+\left|L^{o}_{k-2,k-2}(\sigma)\right|\right).
\end{align*}
By adding $\left|L^{e}_{1,1}(\sigma)\right|+\left|L^{o}_{1,1}(\sigma)\right|=1$ and $\left|L^{e}_{2,2}(\sigma)\right|+\left|L^{o}_{2,2}(\sigma)\right|=3$ to this, we have
\begin{align}\label{eq:ve+vo}
\begin{split}
\left|L^{e}_{k,k}(\sigma)\right|+\left|L^{o}_{k,k}(\sigma)\right|&=F_{2k}\quad(k\ge1).
\end{split}
\end{align}
On the other hand, by (\ref{eq:vrec}), for $k\ge3$ we have
\begin{align*}
&\left|L^{e}_{k,k}(\sigma)\right|-\left|L^{o}_{k,k}(\sigma)\right| \\
&= \left(\left|L^{e}_{k-1,k-1}(\sigma)\right|-\left|L^{o}_{k-1,k-1}(\sigma)\right|\right)+\left(\left|L^{e}_{k-2,k-2}(\sigma)\right|-\left|L^{o}_{k-2,k-2}(\sigma)\right|\right).
\end{align*}
By adding $\left|L^{e}_{1,1}(\sigma)\right|-\left|L^{o}_{1,1}(\sigma)\right|=1$ and $\left|L^{e}_{2,2}(\sigma)\right|-\left|L^{o}_{2,2}(\sigma)\right|=1$ to this, we have
\begin{align}\label{eq:ve-vo}
\begin{split}
\left|L^{e}_{k,k}(\sigma)\right|-\left|L^{o}_{k,k}(\sigma)\right|&=F_{k},\quad(k\ge1).
\end{split}
\end{align}
From (\ref{eq:ve+vo}) and (\ref{eq:ve-vo}), (\ref{eq:vkk}) follows.
\end{proof}
By Theorem \ref{th:Lnumber}, we have $\dim \mathcal{L}^{o}_{k} \le (F_{2k}-F_{k})/2$ and $\dim \mathcal{L}^{e}_{k} \le (F_{2k}+F_{k})/2$ for $k\ge1$. However, these evaluations can be improved by trivial relations \cite[Proposition 2]{U}. For non-negative integers $k$, $d$ and a fixed real number $\sigma$, we define
\begin{align*}
M^{o}_{k,d}(\sigma)=
\left\{\sigma^m\Ls_{k_{1},\dots,k_{n}}^{(l_{1}, \dots, l_{n})}(\sigma) \mathrel{}\middle|\mathrel{} \begin{gathered}m+k_{1}+\dots+k_{n}=k, 0 \le n \le k, \\m\ge0, k_{1},\dots,k_{n} \ge 2, l_{1},\dots,l_{n} \ge 0,\\ k_{1} -1-l_{1},\dots,k_{n} -1-l_{n} \ge 1,\\\sum_{u=1}^{n}k_{u}-1-l_{u} \le d:\text{odd}.\end{gathered}\right\},
\end{align*}
and define $M^{e}_{k,d}(\sigma)$ analogously for the even case. Then, following theorem holds.
\begin{theorem}We have
\begin{align*}
{\rm span}_{\mathbb{Q}}(L^{o}_{k,d}(\sigma))&={\rm span}_{\mathbb{Q}}(M^{o}_{k,d}(\sigma)),\\
{\rm span}_{\mathbb{Q}}(L^{e}_{k,d}(\sigma))&={\rm span}_{\mathbb{Q}}(M^{e}_{k,d}(\sigma)).
\end{align*}
\end{theorem}
\begin{proof}
By applying trivial relations \cite[Proposition 2]{U}, repeatedly, we can see that elements of $L^{o}_{k,d}(\sigma)$ and $L^{e}_{k,d}(\sigma)$ belong to ${\rm span}_{\mathbb{Q}}(M^{o}_{k,d}(\sigma))$ and ${\rm span}_{\mathbb{Q}}(M^{e}_{k,d}(\sigma))$, respectively. Conversely, by $\sigma^m=-m\Ls_{m}^{(m-1)}(\sigma)$ and the shuffle product formula \cite[Proposition 1]{U}, elements of  $M^{o}_{k,d}(\sigma)$ and $M^{e}_{k,d}(\sigma)$ belong to ${\rm span}_{\mathbb{Q}}(L^{o}_{k,d}(\sigma))$ and ${\rm span}_{\mathbb{Q}}(L^{e}_{k,d}(\sigma))$.
\end{proof}
The number of elements of $M^{o}_{k,d}(\sigma)$ and $M^{e}_{k,d}(\sigma)$ are evaluated as follows.
\begin{theorem}\label{th:Mcount}We have
\begin{align}
&\left|M^{o}_{0,d}(\sigma)\right|=\left|M^{o}_{1,d}(\sigma)\right|=0,\ \ &&\left|M^{e}_{0,d}(\sigma)\right|=\left|M^{e}_{1,d}(\sigma)\right|=1,\label{eq:Rinitialk}
\end{align}
and
\begin{align}\label{eq:Rrec}
\begin{split}
\left|M^{o}_{k,d}(\sigma)\right| &=\left|M^{o}_{k-1,d}(\sigma)\right|+\left|M^{e}_{k-1,d-1}(\sigma)\right|\quad(k\ge2,d\ge1),\\
\left|M^{e}_{k,d}(\sigma)\right| &=\left|M^{e}_{k-1,d}(\sigma)\right|+\left|M^{o}_{k-1,d-1}(\sigma)\right|\quad(k\ge2,d\ge1).
\end{split}
\end{align}
In particular, we have
\begin{align}\label{eq:Rkk}
\begin{split}
\left|M^{o}_{0,0}(\sigma)\right| &=\left|M^{o}_{1,1}(\sigma)\right|=0,\quad\left|M^{e}_{0,0}(\sigma)\right|=\left|M^{e}_{1,1}(\sigma)\right|=1,\\
\left|M^{o}_{k,k}(\sigma)\right| &=\left|M^{e}_{k,k}(\sigma)\right|= 2^{k-2}\quad(k\ge2).
\end{split}
\end{align}
\end{theorem}
\begin{proof}
The equation (\ref{eq:Rinitialk}) is clear by the definition.
In order to prove (\ref{eq:Rrec}), we define
\begin{align*}
N^{o}_{k,d}(\sigma)=
\left\{\Ls_{k_{1},\dots,k_{n}}^{(l_{1}, \dots, l_{n})}(\sigma) \mathrel{}\middle|\mathrel{} \begin{gathered}k_{1}+\dots+k_{n}=k, 0 \le n \le k, \\ k_{1},\dots,k_{n} \ge 2, l_{1},\dots,l_{n} \ge 0,\\ k_{1} -1-l_{1},\dots,k_{n} -1-l_{n} \ge 1,\\\sum_{u=1}^{n}k_{u}-1-l_{u} \le d:\text{odd}.\end{gathered}\right\}
\end{align*}
and define $N^{e}_{k,d}(\sigma)$ analogously for the even case.
Then, we have
\begin{align*}
\left|M^{o}_{k,d}(\sigma)\right| &= \sum_{i=0}^{k}\left|N^{o}_{i,d}(\sigma)\right|,\\
\left|M^{e}_{k,d}(\sigma)\right| &= \sum_{i=0}^{k}\left|N^{e}_{i,d}(\sigma)\right|.
\end{align*}
Here, for $k\ge3$ and $d\ge1$, we have
\begin{align*}
\left|N^{o}_{k,d}(\sigma)\right| &=\left|\left\{\Ls_{k_{1},\dots,k_{n}+1}^{(l_{1}, \dots, l_{n}+1)}(\sigma) \mathrel{}\middle|\mathrel{} \Ls_{k_{1},\dots,k_{n}}^{(l_{1}, \dots, l_{n})}(\sigma) \in N^{o}_{k-1,d}(\sigma)\right\}\right|\\
&\quad+\left|\left\{\Ls_{k_{1},\dots,k_{n}+1}^{(l_{1}, \dots, l_{n})}(\sigma) \mathrel{}\middle|\mathrel{} \Ls_{k_{1},\dots,k_{n}}^{(l_{1}, \dots, l_{n})}(\sigma) \in N^{e}_{k-1,d-1}(\sigma)\right\}\right|\\
&\quad+\left|\left\{\Ls_{k_{1},\dots,k_{n},2}^{(l_{1}, \dots, l_{n},0)}(\sigma) \mathrel{}\middle|\mathrel{} \Ls_{k_{1},\dots,k_{n}}^{(l_{1}, \dots, l_{n})}(\sigma) \in N^{e}_{k-2,d-1}(\sigma)\right\}\right|\\
&\quad-\left|\left\{\Ls_{k_{1},\dots,k_{n}+2}^{(l_{1}, \dots, l_{n}+1)}(\sigma) \mathrel{}\middle|\mathrel{} \Ls_{k_{1},\dots,k_{n}}^{(l_{1}, \dots, l_{n})}(\sigma) \in N^{e}_{k-2,d-1}(\sigma)\right\}\right|\\
&=\left|N^{o}_{k-1,d}(\sigma)\right|+\left|N^{e}_{k-1,d-1}(\sigma)\right|.
\end{align*}
Similarly, for $k\ge3$ and $d\ge1$, we obtain
\begin{align*}
\left|N^{e}_{k,d}(\sigma)\right| =\left|N^{e}_{k-1,d}(\sigma)\right|+\left|N^{o}_{k-1,d-1}(\sigma)\right|.
\end{align*}
Therefore, for $k\ge2$ and $d\ge1$, we have
\begin{align*}
\left|M^{o}_{k,d}(\sigma)\right| &= \sum_{i=2}^{k-1}(\left|N^{o}_{i,d}(\sigma)\right|+\left|N^{e}_{i,d-1}(\sigma)\right|)+\left|N^{o}_{2,d}(\sigma)\right|+\left|N^{o}_{1,d}(\sigma)\right|+\left|N^{o}_{0,d}(\sigma)\right|\\
&=\left|M^{o}_{k-1,d}(\sigma)\right|+\left|M^{e}_{k-1,d-1}(\sigma)\right|-\left|N^{e}_{1,d-1}(\sigma)\right|-\left|N^{e}_{0,d-1}(\sigma)\right|+\left|N^{o}_{2,d}(\sigma)\right|\\
&=\left|M^{o}_{k-1,d}(\sigma)\right|+\left|M^{e}_{k-1,d-1}(\sigma)\right|,
\end{align*}
and
\begin{align*}
\left|M^{e}_{k,d}(\sigma)\right| &= \sum_{i=2}^{k-1}(\left|N^{e}_{i,d}(\sigma)\right|+\left|N^{o}_{i,d-1}(\sigma)\right|)+\left|N^{e}_{2,d}(\sigma)\right|+\left|N^{e}_{1,d}(\sigma)\right|+\left|N^{e}_{0,d}(\sigma)\right|\\
&=\left|M^{e}_{k-1,d}(\sigma)\right|+\left|M^{o}_{k-1,d-1}(\sigma)\right|-\left|N^{o}_{1,d-1}(\sigma)\right|-\left|N^{o}_{0,d-1}(\sigma)\right|+\left|N^{e}_{2,d}(\sigma)\right|\\
&=\left|M^{e}_{k-1,d}(\sigma)\right|+\left|M^{o}_{k-1,d-1}(\sigma)\right|
\end{align*}
which prove (\ref{eq:Rrec}).
By (\ref{eq:Rrec}), for $k\ge2$, we obtain
\begin{align*}
\left|M^{e}_{k,k}(\sigma)\right|+\left|M^{o}_{k,k}(\sigma)\right|=2(\left|M^{e}_{k-1,k-1}(\sigma)\right|+\left|M^{o}_{k-1,k-1}(\sigma)\right|).
\end{align*}
By adding $\left|M^{e}_{1,1}(\sigma)\right|+\left|M^{o}_{1,1}(\sigma)\right|=1$ to this, for $k\ge1$, we obtain
\begin{align*}
\left|M^{e}_{k,k}(\sigma)\right|+\left|M^{o}_{k,k}(\sigma)\right|=2^{k-1}.
\end{align*}
On the other hand, by (\ref{eq:Rrec}), for $k\ge2$, we obtain
\begin{align*}
\left|M^{e}_{k,k}(\sigma)\right|-\left|M^{o}_{k,k}(\sigma)\right|=0.
\end{align*}
Therefore, we can obtain (\ref{eq:Rkk}).
\end{proof}
By Theorem \ref{th:Mcount}, we can obtain $\dim \mathcal{L}^{o}_{k} \le2^{k-2}$ and $\dim \mathcal{L}^{e}_{k} \le 2^{k-2}$ for $k\ge2$.
\begin{conjecture}\label{con:genL}The following assertions hold.
\begin{enumerate}[(i)]
\item Every iterated log-sine integral at $\pi/3$ of weight $k$ and odd $d:=|\mathbf{k}-\mathbf{1}_{n}-\mathbf{l}|$ can be written as a $\mathbb{Q}$-linear combination of elements of $S^{o}_{k,d}$.
\item Every iterated log-sine integral at $\pi/3$ of weight $k$ and even $d:=|\mathbf{k}-\mathbf{1}_{n}-\mathbf{l}|$ can be written as a $\mathbb{Q}$-linear combination of elements of $S^{e}_{k,d}$.
\end{enumerate}
\end{conjecture}
\begin{conjecture}\label{con:basisL}The following assertions hold.
\begin{enumerate}[(i)]
 \item The set of the real numbers $S^{o}_{k,k}$ is a basis of $\mathcal{L}^{o}_{k}$.
 \item The set of the real numbers $S^{e}_{k,k}$ is a basis of $\mathcal{L}^{e}_{k}$.
 \item The set of the real numbers $S^{o}_{k,k}\cup S^{e}_{k,k}$ is a basis of $\mathcal{L}^{o}_{k}+\mathcal{L}^{e}_{k}$.
\end{enumerate}
\end{conjecture}
If Conjecture \ref{con:basisL} is true, then $\dim \mathcal{L}^{o}_{k}=I(k)$ and $\dim \mathcal{L}^{e}_{k}=R(k)$ hold.

The following is a table on the dimensions of $\mathcal{L}^{o}_{k}$.
\begin{table}[htbp]
\centering
  \begin{tabular}{|c|c|c|c|c|c|c|c|c|c|c|c|c} \hline
     $k$& 0 & 1 & 2 & 3 & 4 &5 &6 &7&8&9&10&$\cdots$\\ \hline 
    $|L^{o}_{k,k}(\sigma)|$ &0&0& 1 & 3 & 9 &25&68&182&483&1275&3355&$\cdots$\\ \hline 
    $|M^{o}_{k,k}(\sigma)|$&0&0 & 1 & 2 & 4 &8&16&32&64&128&256&$\cdots$\\ \hline
    $I(k)$&0& 0& 1 & 2 & 3 &4&6&10&17&28&45&$\cdots$\\ \hline
  \end{tabular}
\end{table}\\
Here, $|L^{o}_{k,k}(\sigma)|$, $|M^{o}_{k,k}(\sigma)|$, and $I(k)$ are the number of generators of the definition of $\mathcal{L}^{o}_{k}$, the upper bound of $\dim \mathcal{L}^{o}_{k}$ given by trivial relations \cite[Proposition 2]{U}, and the conjectured dimension of $\mathcal{L}^{o}_{k}$, respectively.

The following table is the even version of the table above.
\begin{table}[htbp]
\centering
  \begin{tabular}{|c|c|c|c|c|c|c|c|c|c|c|c|c} \hline
     $k$& 0 & 1 & 2 & 3 & 4 &5 &6 &7&8&9&10&$\cdots$\\ \hline 
    $|L^{e}_{k,k}(\sigma)|$ &1&1& 2 & 5 & 12 &30&76&195&504&1309&3410&$\cdots$\\ \hline 
    $|M^{e}_{k,k}(\sigma)|$&1&1 & 1 & 2 & 4 &8&16&32&64&128&256&$\cdots$\\ \hline
    $R(k)$&1& 1& 1 & 1 & 2 &4&7&11&17&27&44&$\cdots$\\ \hline
  \end{tabular}
\end{table}

Let $|L_{k,k}(\sigma)| = |L^{o}_{k,k}(\sigma)| + |L^{e}_{k,k}(\sigma)|$ and $|M_{k,k}(\sigma)| = |M^{o}_{k,k}(\sigma)| + |M^{e}_{k,k}(\sigma)|$. Then, a table on the dimensions of $\mathcal{L}^{o}_{k}+\mathcal{L}^{e}_{k}$ is as follows.
\newpage
\begin{table}[htbp]
\centering
  \begin{tabular}{|c|c|c|c|c|c|c|c|c|c|c|c|c} \hline
     $k$& 0 & 1 & 2 & 3 & 4 &5 &6 &7&8&9&10&$\cdots$\\ \hline 
    $|L_{k,k}(\sigma)|$ &1&1& 3 & 8 & 21 &55&144&377&987&2584&6765&$\cdots$\\ \hline 
    $|M_{k,k}(\sigma)|$&1 & 1 & 2 & 4 &8&16&32&64&128&256&512&$\cdots$\\ \hline
    $W(k)$& 1& 1 & 2 & 3 &5&8&13&21&34&55&89&$\cdots$\\ \hline
  \end{tabular}
\end{table}

Conjecture \ref{con:genL} follows from Conjecture \ref{con:genZ} and Conjecture \ref{con:genCG} as follows. Therefore, Conjecture \ref{con:genL} can be checked numerically the weight up to $9$, indirectly.
\begin{theorem}\label{th:ZCGtoL}
If Conjecture \ref{con:genZ} and Conjecture \ref{con:genCG} are true, then Conjecture \ref{con:genL} is true.
\end{theorem}
\begin{proof}
Let $\mathcal{S}^{o}_{k,d}=\text{span}_{\mathbb{Q}}(S^{o}_{k,d})$ and  $\mathcal{S}^{e}_{k,d}=\text{span}_{\mathbb{Q}}(S^{e}_{k,d})$.
We show that the right hand side of (\ref{eq:main}) with $\sigma=\pi/3$ belongs to $\mathcal{S}^{o}_{|\mathbf{k}|,|\mathbf{k}-\mathbf{1}_{n}-\mathbf{l}|}$ if $|\mathbf{k}-\mathbf{1}_{n}-\mathbf{l}|$ is odd, and belongs to  $\mathcal{S}^{e}_{|\mathbf{k}|,|\mathbf{k}-\mathbf{1}_{n}-\mathbf{l}|}$ if $|\mathbf{k}-\mathbf{1}_{n}-\mathbf{l}|$ is even assuming Conjecture \ref{con:genZ} and Conjecture \ref{con:genCG}. By (\ref{eq:f(0)}) and Conjecture \ref{con:genZ}, we obtain 
\begin{align*}
f_{\mathbf{q}}^{\mathbf{r}}(0) \in \mathcal{S}^{\prime}_{|\mathbf{q}|+|\mathbf{r}|+n,|\mathbf{r}|}.
\end{align*}
Therefore, we obtain
\begin{align}
&\prod_{j=1}^{h-1}f_{\mathbf{q}^{(j)}}^{\mathbf{r}^{(j)}}(0)\label{eq:th10,1} \\
&\in \mathcal{S}^{\prime}_{\sum_{j=1}^{h-1}|\mathbf{q}^{(j)}| + |\mathbf{r}^{(j)}| + {\rm dep}(\mathbf{r}^{(j)}), \sum_{j=1}^{h-1}|\mathbf{r}^{(j)}|}\nonumber\\
&\subset\begin{cases}\mathcal{S}^{o}_{\sum_{j=1}^{h-1}|\mathbf{q}^{(j)}| + |\mathbf{r}^{(j)}| + {\rm dep}(\mathbf{r}^{(j)}), \sum_{j=1}^{h-1}|\mathbf{r}^{(j)}|}\\
\qquad\qquad\qquad\qquad\qquad(\sum_{j=1}^{h-1}|\mathbf{q}^{(j)}| + |\mathbf{r}^{(j)}| + {\rm dep}(\mathbf{r}^{(j)}):{\rm odd}),\\
\mathcal{S}^{e}_{\sum_{j=1}^{h-1}|\mathbf{q}^{(j)}| + |\mathbf{r}^{(j)}| + {\rm dep}(\mathbf{r}^{(j)}), \sum_{j=1}^{h-1}|\mathbf{r}^{(j)}|}\\
\qquad\qquad\qquad\qquad\qquad(\sum_{j=1}^{h-1}|\mathbf{q}^{(j)}| + |\mathbf{r}^{(j)}| + {\rm dep}(\mathbf{r}^{(j)}):{\rm even}).\end{cases}\nonumber
\end{align}
On the other hand, by Conjecture \ref{con:genCG} we have
\begin{align*}
L(w_{\mathbf{r}''}^{\mathbf{j}};e^{i\pi/3})\in\begin{cases} \mathcal{S}^{o}_{|\mathbf{r}|+|\mathbf{j}|+n'',|\mathbf{r}|}+i\mathcal{S}^{e}_{|\mathbf{r}|+|\mathbf{j}|+n'',|\mathbf{r}|}\quad&|\mathbf{r}|+|\mathbf{j}|+n'':{\rm odd},\\
\mathcal{S}^{e}_{|\mathbf{r}|+|\mathbf{j}|+n'',|\mathbf{r}|}+i\mathcal{S}^{o}_{|\mathbf{r}|+|\mathbf{j}|+n'',|\mathbf{r}|}\quad&|\mathbf{r}|+|\mathbf{j}|+n'':{\rm even}.\end{cases}
\end{align*}
By multiplying this expression by $(i\pi/3)^{|\overline{\mathbf{q}}|-|\mathbf{j}|}$, we obtain
\begin{align*}
(i\pi/3)^{|\overline{\mathbf{q}}|-|\mathbf{j}|}L(w_{\mathbf{r}''}^{\mathbf{j}};e^{i\pi/3})\in\begin{cases} \mathcal{S}^{o}_{|\mathbf{q}|+|\mathbf{r}|+n,|\mathbf{r}|}+i\mathcal{S}^{e}_{|\mathbf{q}|+|\mathbf{r}|+n,|\mathbf{r}|}\quad&|\mathbf{q}|+|\mathbf{r}|+n:{\rm odd},\\
\mathcal{S}^{e}_{|\mathbf{q}|+|\mathbf{r}|+n,|\mathbf{r}|}+i\mathcal{S}^{o}_{|\mathbf{q}|+|\mathbf{r}|+n,|\mathbf{r}|}\quad&|\mathbf{q}|+|\mathbf{r}|+n:{\rm even}.\end{cases}
\end{align*}
Therefore, we obtain
\begin{align}
&f_{\mathbf{q}^{(h)}}^{\mathbf{r}^{(h)}}(\pi/3)-f_{\mathbf{q}^{(h)}}^{\mathbf{r}^{(h)}}(0)\label{eq:th10,2} \\
&\in\begin{cases} \mathcal{S}^{o}_{|\mathbf{q}^{(h)}|+|\mathbf{r}^{(h)}|+{\rm dep}(\mathbf{q}^{(j)}),|\mathbf{r}^{(h)}|}+i\mathcal{S}^{e}_{|\mathbf{q}^{(h)}|+|\mathbf{r}^{(h)}|+{\rm dep}(\mathbf{q}^{(j)}),|\mathbf{r}^{(h)}|}\\
\qquad\qquad\qquad\qquad\qquad\qquad(|\mathbf{q}^{(h)}|+|\mathbf{r}^{(h)}|+{\rm dep}(\mathbf{q}^{(j)}):{\rm odd}),\\
\mathcal{S}^{e}_{|\mathbf{q}^{(h)}|+|\mathbf{r}^{(h)}|+{\rm dep}(\mathbf{q}^{(j)}),|\mathbf{r}^{(h)}|}+i\mathcal{S}^{o}_{|\mathbf{q}^{(h)}|+|\mathbf{r}^{(h)}|+{\rm dep}(\mathbf{q}^{(j)}),|\mathbf{r}^{(h)}|}\\
\qquad\qquad\qquad\qquad\qquad\qquad(|\mathbf{q}^{(h)}|+|\mathbf{r}^{(h)}|+{\rm dep}(\mathbf{q}^{(j)}):{\rm even}).\end{cases}\nonumber
\end{align}
By (\ref{eq:th10,1}) and (\ref{eq:th10,2}), we obtain
\begin{align*}
F_{\mathbf{q}}^{\mathbf{r}}(\pi/3)\in\begin{cases} \mathcal{S}^{o}_{|\mathbf{q}|+|\mathbf{r}|+n,|\mathbf{r}|}+i\mathcal{S}^{e}_{|\mathbf{q}|+|\mathbf{r}|+n,|\mathbf{r}|}
\quad&|\mathbf{q}|+|\mathbf{r}|+n:{\rm odd},\\
\mathcal{S}^{e}_{|\mathbf{q}|+|\mathbf{r}|+n,|\mathbf{r}|}+i\mathcal{S}^{o}_{|\mathbf{q}|+|\mathbf{r}|+n,|\mathbf{r}|}
\quad&|\mathbf{q}|+|\mathbf{r}|+n:{\rm even}.\end{cases}
\end{align*}
Noting $\mathbf{p}+\mathbf{q}+\mathbf{r}=\mathbf{k}-\mathbf{1}_{n}-\mathbf{l}$, we obtain
\begin{align*}
i^{|\mathbf{l}|+n}(-i\pi)^{|\mathbf{p}|}F_{\mathbf{q}+\mathbf{l}}^{\mathbf{r}}(\pi/3)\in\begin{cases} \mathcal{S}^{o}_{|\mathbf{k}|,|\mathbf{r}|}+i\mathcal{S}^{e}_{|\mathbf{k}|,|\mathbf{r}|}
\quad&|\mathbf{k}-\mathbf{1}_{n}-\mathbf{l}|:{\rm odd},\\
\mathcal{S}^{e}_{|\mathbf{k}|,|\mathbf{r}|}+i\mathcal{S}^{o}_{|\mathbf{k}|,|\mathbf{r}|}
\quad&|\mathbf{k}-\mathbf{1}_{n}-\mathbf{l}|:{\rm even}.\end{cases}
\end{align*}
Since $|\mathbf{r}|$ takes the maximum value $|\mathbf{k}-\mathbf{1}_{n}-\mathbf{l}|$ when $|\mathbf{p}|=|\mathbf{q}|=0$, we obtain
\begin{align*}
\Ls_{\mathbf{k}}^{\mathbf{l}}(\sigma)\in\begin{cases} \mathcal{S}^{o}_{|\mathbf{k}|,|\mathbf{k}-\mathbf{1}_{n}-\mathbf{l}|}\quad&|\mathbf{k}-\mathbf{1}_{n}-\mathbf{l}|:{\rm odd},\\
\mathcal{S}^{e}_{|\mathbf{k}|,|\mathbf{k}-\mathbf{1}_{n}-\mathbf{l}|}\quad&|\mathbf{k}-\mathbf{1}_{n}-\mathbf{l}|:{\rm even}.\end{cases}
\end{align*}
\end{proof}
\section*{Acknowledgment}
The author is deeply grateful to Prof.\ Kohji Matsumoto, Prof.\ Koji Tasaka 
and all those attending the seminar at Kindai University, including Prof.\ Kentaro Ihara, Prof.\ Yayoi Nakamura and Prof.\ Yasuo Ohno for their helpful comments. He is also deeply grateful to Dr.\ Minoru Hirose for providing the author a computer program for numerical evaluation of multiple zeta values and multiple polylogarithms.

\vspace{4mm}

{\footnotesize
{\sc
\noindent
Graduate School of Mathematics, Nagoya University,\\
Chikusa-ku, Nagoya 464-8602, Japan.
}\\
{\it E-mail address}, R. Umezawa\hspace{1.75mm}: {\tt
m15016w@math.nagoya-u.ac.jp}\\
}

\end{document}